\documentclass[final]{siamltex}

\usepackage{amsfonts}
\usepackage{graphicx}
\usepackage{amsmath}
\usepackage{graphicx}
\usepackage{amssymb}
\usepackage{pdfsync}
\usepackage[colorlinks,linkcolor={blue}]{hyperref}%
\newtheorem{remark}[theorem]{Remark}

\usepackage{color}    
\definecolor{Red}{rgb}{1,0,0}

\usepackage{ulem}    
\usepackage{url}    

\DeclareGraphicsExtensions{.eps,.pdf,{}}

\title{Validated Computation of Heteroclinic Sets}

\author{Maciej J. Capi\'nski \thanks{AGH University of Science and Technology, Faculty of Applied Mathematics, al. Mickiewicza 30, 30-059 Krak\'ow. Email: {\tt mcapinsk@agh.edu.pl}. The work was partially supported by the Polish National Science Center grant 2012/05/B/ST1/00355.
} 
\and 
J.D. Mireles James \thanks{Florida Atlantic University, Department of Mathematical Sciences, 777 Glades Rd, Boca Raton, FL 33431, United States. Email: {\tt jmirelesjames@fau.edu}. The work was partially supported by NSF grant DMS - 1318172.}}

\begin{document}
\maketitle
\begin{abstract}
In this work we develop a method for computing mathematically
rigorous enclosures of some one dimensional manifolds of
heteroclinic orbits for nonlinear maps. Our method exploits a 
rigorous curve following argument build on high order Taylor 
approximation of the local stable/unstable manifolds.  
The curve following argument is a uniform interval Newton method 
applied on short line segments.  The definition of the heteroclinic sets 
involve compositions of the map and we 
use a Lohner-type representation to overcome the accumulation of
roundoff errors.  Our argument requires precise control over the 
local unstable and stable manifolds so that we must 
first obtain validated a-posteriori error bounds on the truncation errors
associated with the manifold approximations.  We illustrate the utility
of our method by proving some computer assisted theorems about
heteroclinic invariant sets for a volume preserving map of $\mathbb{R}^3$. 
\end{abstract}



\pagestyle{myheadings}
\thispagestyle{plain}
\markboth{M.J. CAPI\'NSKI AND J.D. MIRELES JAMES }{VALIDATED COMPUTATION OF HETEROCLINIC SETS}



\section{Introduction}
The study of a nonlinear dynamical system
begins by considering the fixed/periodic points, 
their linear stability, and their local stable and 
unstable manifolds. In order to patch this local information 
into a global picture of the dynamics one then wants to 
understand the connecting orbits. Questions about 
connecting orbits are naturally reformulated as
questions about where and how the stable and unstable 
manifolds intersect.

A typical situation is that we consider a hyperbolic fixed point, so that the
dimension of the stable manifold plus the dimension of the unstable manifold
add up to the dimension of the whole space. A transverse
intersection of the stable/unstable manifolds of a hyperbolic
fixed point is again a single point. Such
intersections give rise to homoclinic connecting orbits, and they are of
special interest. For example Smale's tangle theorem says that the existence
of a transverse homoclinic intersection point implies the existence of
hyperbolic chaotic motions \cite{MR0228014}.

More generally, consider a pair of distinct fixed points and assume
that the dimensions $s$ and $u$ of their stable and unstable manifolds
have $s+u=d>k$, with $k$ being the
dimension of the ambient space. In this case a transverse intersection of
the manifolds results in a set which is (locally) a $d-k$
dimensional manifold of connecting orbits. 
For example we are sometimes interested in transport
barriers or separatrices formed by codimension one stable or unstable
manifolds \cite{2015Chaos..25i7602M}.

\begin{figure}[t] 
\begin{center}
\scalebox{0.175}{\includegraphics{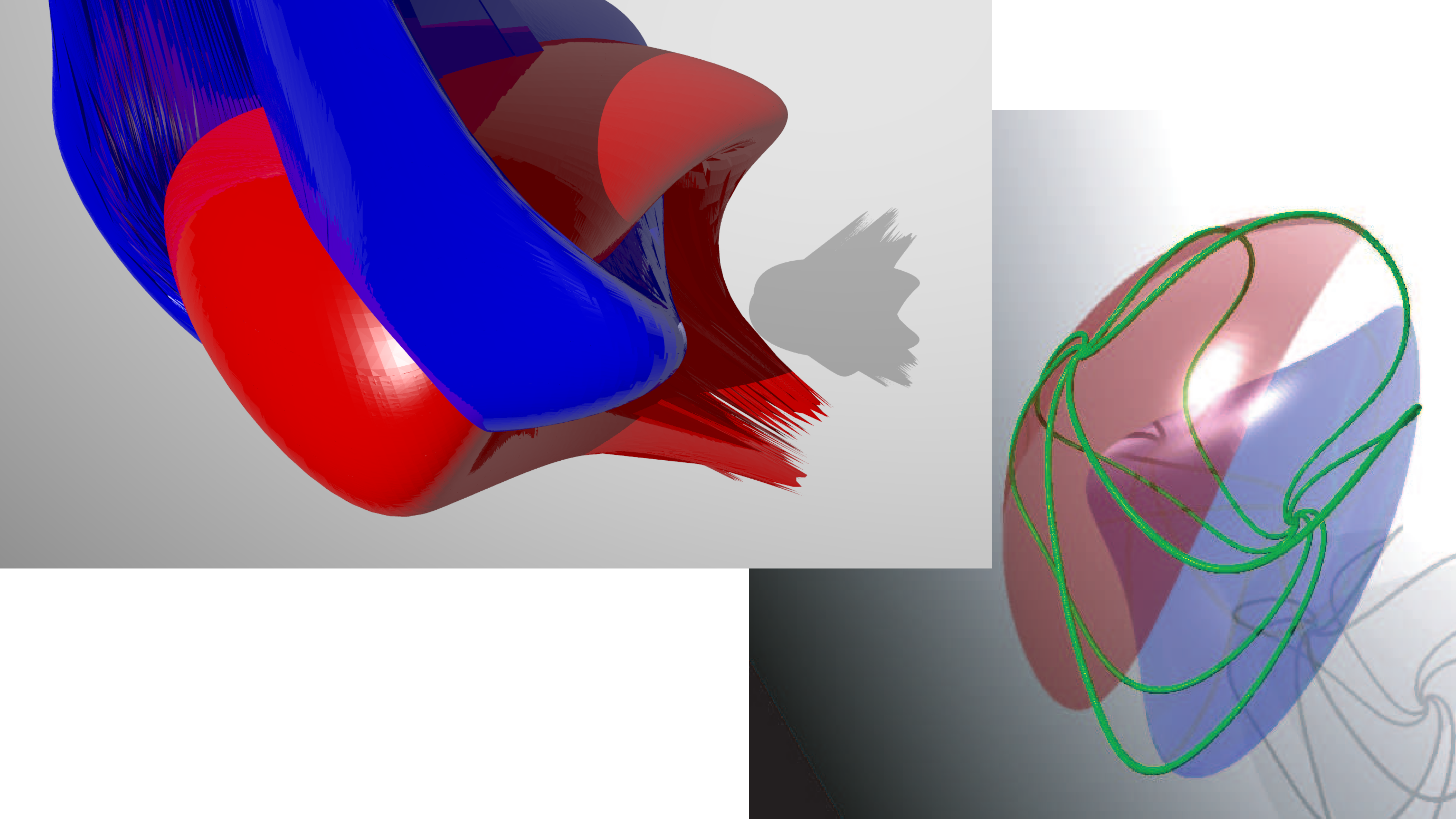}}
\end{center}
\caption{Heteroclinic intersection arcs for the Lomel\'{\i} map 
(see Equation \eqref{eq:LomeliMap}).  Unstable manifolds in 
blue, stable manifolds in red, intersection curves in green.
The intersection of the manifolds yields six distinct curves.
Every point on a curve is a heteroclinic orbit.
(All references to color refer to the online version). 
}  \label{fig:LomeliMap1}
\end{figure}

\begin{figure}[t] 
\begin{center}
\scalebox{0.175}{\includegraphics{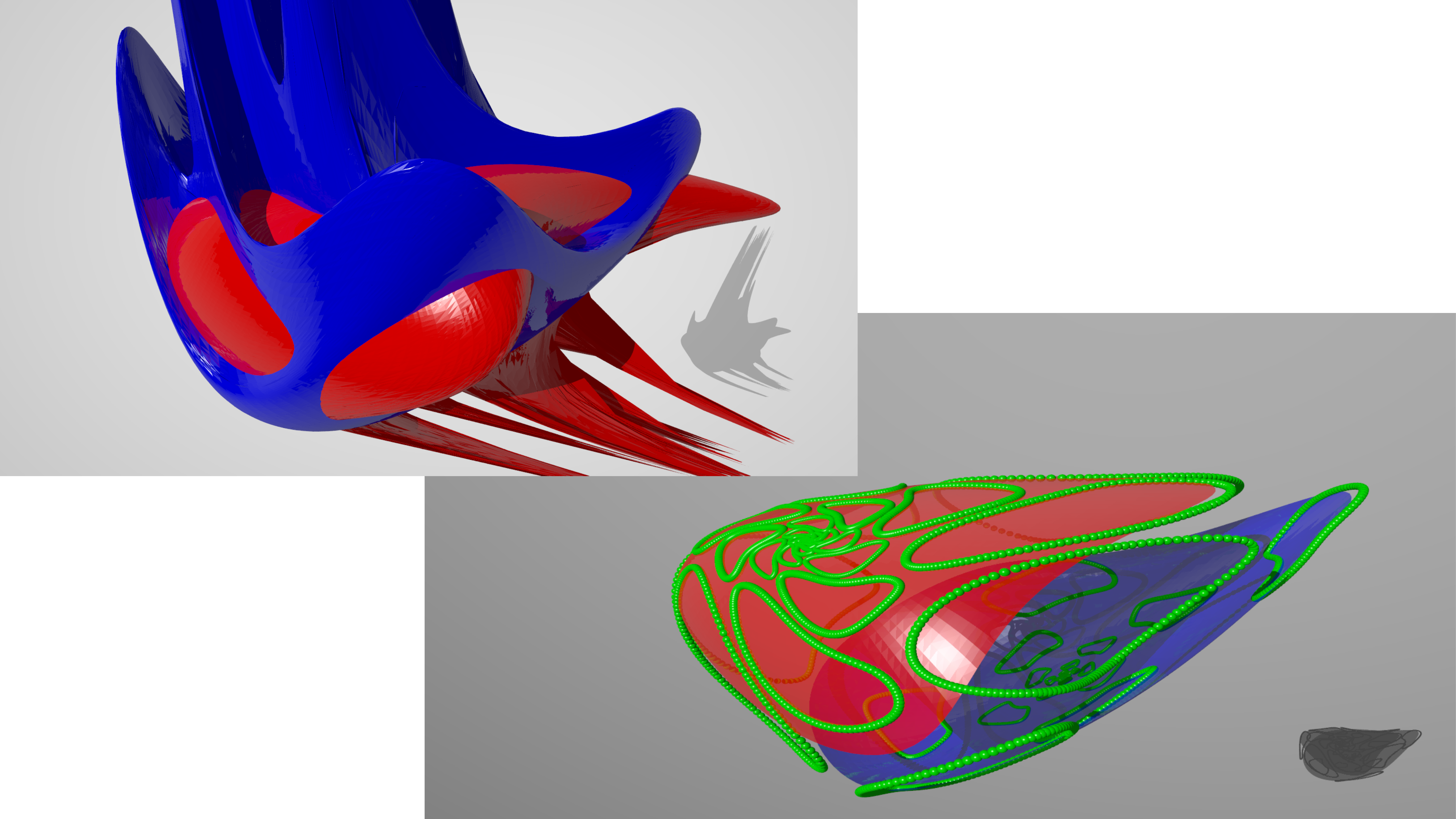}}
\end{center}
\caption{Heteroclinic intersection loops for the Lomel\'{\i} map 
(see Equation \eqref{eq:LomeliMap}).  Unstable manifolds in blue, 
stable manifolds in red, intersection curves in green.  The intersection
of the manifolds yields a countable system of loops.  Every point on a
loop is a heteroclinic orbit. 
(All references to color refer to the online version).
} \label{fig:LomeliMap2}
\end{figure}

In the present work we  develop a computer assisted argument for proving the
existence of one dimensional intersections between stable/unstable
manifolds of distinct fixed points. Our arguments utilize the underlying
dynamics of the map in order to obtain the alpha/omega limit sets of the
intersection manifolds. In addition to abstract existence results our argument 
yields error bounds on the location of the intersection in phase space. 
Moreover we see that transversality of the intersections follows as a natural 
corollary, so that the heteroclinic intersections we obtain are in fact normally hyperbolic 
invariant manifolds.

The two main ingredients in our argument are high order numerical 
computation of the local stable/unstable manifolds with mathematically 
rigorous bounds on the truncation error, and a validated branch following algorithm  
used to rigorously enclose the intersections of these manifolds.
Our treatment of the stable/unstable manifolds is based on the 
Parameterization method for invariant manifolds 
\cite{MR1976079, MR1976080, MR2177465,  MR2240743, MR2289544}.
The Parameterization Method is 
a general functional analytic framework for studying invariant manifolds
which reduces questions about the 
manifold to questions about the solutions of a certain 
nonlinear operator equation.  
Reformulating the problem in terms of an operator equation facilitates 
the design of efficient numerical algorithms for computing 
the manifold, and introduces a notion of 
a-posteriori error for the numerical approximations.  
 A-posteriori analysis of the operator leads to computer
assisted bounds on the truncation errors.

Next we use these parametric representations of the local stable/unstable 
manifolds to recast the desired 
heteroclinic intersection as the one-dimensional zero set of a certain 
finite dimensional map.  An approximate zero set is computed 
numerically via a Newton/continuation scheme, and a uniform 
Newton-Krawczyk argument is applied along the $d-k$ dimensional 
approximate zero set yields a mathematically rigorous enclosure
of the true solution.  This validated branch following is complicated 
by the fact that the finite dimensional map contains terms
given by multiple compositions of the underlying
nonlinear dynamical system. A carefully chosen coordinate 
system is defined along the branch which allows us to mitigate 
the so called ``wrapping effect''.
We mention that a number of other authors have developed methods
for validated computation of zero sets, and refer the 
interested reader to the works of 
 \cite{ MR2679365, MR2539195, MR2338393,  galepu, MR2326238, MR2630003}, 
 and the references discussed therein for a more 
complete overview of the literature.

We illustrate our method in a concrete example, and study 
some global heteroclinic sets in $\mathbb{R}^3$,
which are defined as the transverse intersection
of the unstable and stable manifolds for a pair of distinct
fixed points of the \textit{Lomel\'{\i} map}.  This
map is defined by Equation \eqref{eq:LomeliMap} in 
Section \ref{sec:vortexBubbles}, where we discuss the 
map and its dynamics in more detail.  Figures 
\ref{fig:LomeliMap1} and \ref{fig:LomeliMap2} 
illustrate the unstable and stable manifolds, as well as  
their intersections, for two different choices of parameters
for the Lomel\'{\i} map.  

Figure \ref{fig:LomeliMap1} illustrates the map 
with parameter values $a = 0.44$, $b = 0.21$,
$c = 0.35$, $\alpha = -0.25$, and $\tau = -0.3$.
The map with these parameters was also studied in 
\cite{MR2728178}.  The upper left frame suggests that the intersection
of $W^u(p_1)$ and $W^s(p_2)$ is a system of
arcs beginning and ending at the fixed points.
The numerically computed intersection of the manifolds is shown 
as a collection of green points in the lower right frame.  In 
Section \ref{sec:intersections} (see Section \ref{sec:arc-CAP}, in particular)
we prove the existence of two 
distinct intersection arcs whose iterates generate all six 
curves shown in the lower right frame.  

Figure \ref{fig:LomeliMap2} illustrates the map
with parameter values $a = 0.5$, $b = -0.5$,
$c = 1$, $\alpha = -0.08999$, and $\tau =  0.8$. Again,
the map with these parameters was also studied in
Section $4$ of \cite{MR3079670}, see
especially the bottom left frame of figure $4.5$ in that reference.
For these parameter values the heteroclinic intersections
of $W^u(p_1)$ and $W^s(p_2)$ is a system of loops,
as we see islands of red surrounded by blue in the 
upper left frame.  The lower right frame illustrates the numerically 
computed intersection of the manifolds as a collection of 
green points.  In Section \ref{sec:intersections} (see Section \ref{sec:loop-CAP}, in particular) we 
prove the existence of a single intersection loop
whose iterates generate the system of loops shown in the lower right 
frame.

The methods of the present work 
provide computer assisted proof that the heteroclinic invariant 
sets suggested by these pictures actually exist.
While we implement our methods only for intersection 
arcs for the Lomel\'{\i} family in $\mathbb{R}^3$, 
it is clear that the theoretical framework developed here applies 
much more generally.  Indeed, employing the rigorous numerical
methods for multiparameter continuation recently developed in 
\cite{galepu} it should be possible to adapt our methods 
to the study of higher dimensional manifold intersections.

The remainder of the paper is organized as follows.  In Section 
\ref{sec:prelim} we establish some notation and review some 
preliminary material including 
the definitions and basic properties of the stable and unstable 
manifolds of fixed points, the definitions and basic properties
of heteroclinic invariant sets.  We also discuss the dynamics 
of the Lomel\'{\i} map.  In Section \ref{sec:parmMeth} we review the 
basic notions of the Parameterization Methods for stable/unstable
manifolds of fixed points, and illustrate the formal computation 
of the Taylor expansion of the parameterization.  We recall an 
a-posteriori validation theorem which allows us to 
obtain rigorous computer assisted bounds on the 
truncation errors.  Finally in Section \ref{sec:intersections} we 
develop and implment the main tools of the paper, namely the curve 
following/continuation argument used to enclose the heteroclinic
arcs.  All of the codes used to obtain the results in this paper can 
be found at the web page \cite{ourHeteroArcCodes}.

\section{Preliminaries}

\label{sec:prelim}
\subsection{Notations}
Throughout this paper we shall use $B_{k}(x,R)$ to denote a ball of radius $R$
about $x \in\mathbb{R}^{k}$. To simplify notations we will also write
$B_{k}(R)$ for a ball centered at zero, and $B_{k}$ for a ball centered at
zero with radius $1$. Similarly, for $x + iy \in\mathbb{C}$ let $|z| =
\sqrt{x^{2} + y^{2}}$ denote the usual absolute value for complex numbers and
for $z = (z_{1}, \ldots, z_{k}) \in\mathbb{C}^{k}$ define the norm
\[
\| z \|_{\mathbb{C}^{k}} = \max_{1 \leq i \leq k} | z_{i} |.
\]
For $R > 0$ and $z \in\mathbb{C}^{k}$ let
\[
D_{k}(z, R) = \{w \in\mathbb{C}^{k} : \| w - z \|_{\mathbb{C}^{k}} < R \},
\]
denote the poly-disk of radius $R > 0$ about $z \in\mathbb{C}^{k}$. We write
$D_{k}(R)$ for the ball centered at zero and $D_{k}$ for the unit ball
centered at zero.

If we consider $p=(x,y)$, then we will use $\pi_x p$, $\pi_y p$ for the projections onto the $x,y$ coordinates, respectively.


We now write out the interval arithmetic notations conventions that will be
used in the paper. 
Let $U$ be a subset of $\mathbb{R}^{k}$. We shall denote by $[U]$ an interval
enclosure of the set $U$, that is, a set
\[
\lbrack U]=\Pi_{i=1}^{k}[a_{i},b_{i}]\subset\mathbb{R}^{k},
\]
such that $U\subset\lbrack U]$. Similarly, for a family of matrixes
$\mathbf{A}\subset\mathbb{R}^{k\times m}$ we denote its interval enclosure as
$\left[  \mathbf{A}\right]  $, that is, a set
\[
\left[  \mathbf{A}\right]  =\left(  [a_{ij},b_{ij}]\right)
_{\substack{i=1,...,k\\j=1,...,m}}\subset\mathbb{R}^{k\times m},
\]
such that $\mathbf{A}\subset\left[  \mathbf{A}\right]  $. For $F:\mathbb{R}%
^{k}\rightarrow\mathbb{R}^{m}$, by $[DF(U)]$ we shall denote an interval
enclosure
\[
\lbrack DF(U)]=\left[\left\{  A\in\mathbb{R}^{k\times m}|A_{ij}\in\left[
\inf_{x\in U}\frac{\partial F_{i}}{\partial x_{j}}(x),\sup_{x\in U}%
\frac{\partial F_{i}}{\partial x_{j}}(x)\right]  \right\}\right]  .
\]
For a set $U$ and a family of matrixes $\mathbf{A}$ we shall use the
notation $\left[  \mathbf{A}\right]  \left[  U\right]  $ to denote an interval
enclosure
\[
\left[  \mathbf{A}\right]  \left[  U\right]  =\left[  \left\{  Au:A\in\left[
\mathbf{A}\right]  ,u\in\left[  U\right]  \right\}  \right]  .
\]
We shall say that a family of matrixes $\mathbf{A}\subset\mathbb{R}^{k\times
k}$ is invertible, if each matrix $A\in\mathbf{A}$ is invertible. We shall
also use the notation%
\[
\left[  \mathbf{A}\right]  ^{-1}\left[  U\right]  =\left[  \left\{
A^{-1}u:A\in\left[  \mathbf{A}\right]  ,u\in\left[  U\right]  \right\}
\right]  .
\]

\subsection{Stable/unstable manifolds of fixed points}

\label{sec:stableMan} The material in this section is standard and can be
found in any textbook on the qualitative theory of dynamical systems. We refer
for example to \cite{MR1326374, MR1396532}. Let $f \colon\mathbb{C}^{k}
\to\mathbb{C}^{k}$ be a smooth (in our case analytic) map and assume that $p
\in\mathbb{C}^{k}$ is a fixed point of $f$. Let $U \subset\mathbb{C}^{k}$ be a
neighborhood of $p$, and define the set
\[
W_{\mathrm{loc}}^{s}(p, U) = \{w \in U : f^{n}(w) \in U \, \mbox{for all } n
\geq0\}.
\]
This set is referred to as the local stable set of $p$ relative to $U$.

If $p$ is a hyperbolic fixed point, i.e. if none of the eigenvalues of $Df(p)$
are on the unit circle, then the stable manifold theorem gives that there
exists a neighbourhood $U \subset\mathbb{C}^{k}$ of $p$ so that
$W_{\mathrm{loc}}^{s}(p, U)$ is an $m$-dimensional embedded disk tangent to
the stable eigenspace at $p$. If $f$ is analytic then the embedding is analytic.

The invariant set
\begin{align*}
W^{s}(p)  & = \left\{ w \in\mathbb{R}^{n} : \lim_{n\to\infty} f^{n}(w) = p
\right\} ,
\end{align*}
is the stable manifold of $p$, and consists of all orbits which accumulate
under forward iteration of the map to the fixed point $p$. If $f$ is
invertible then we have that
\[
W^{s}(p) = \bigcup_{n=0}^{\infty}f^{-n}\left[ W_{\mathrm{loc}}^{s}(p,U)\right]
,
\]
i.e. the stable manifold is obtained as the union of all backwards iterates of
a local stable manifold.

We say that the sequence $\{x_{j}\}_{j=-\infty}^{0}$ is a backward orbit of
$x_{0}$ if
\[
f(x_{j}) = x_{j+1},
\]
for all $j \leq-1$. If $U \subset\mathbb{R}^{k}$ and $x_{j} \in U$ for all $j
\leq0$ we say that $x_{0}$ has a backward orbit in $U$. If
\[
\lim_{j \to-\infty} x_{j} = p,
\]
we say that $x_{0}$ has a backward orbit accumulating at $p$. Let $U
\in\mathbb{C}^{k}$ be an open neighborhood of $p$ and define the set
\[
W_{\mathrm{loc}}^{u}(p, U) = \{w \in U \colon w \mbox{ has a backward orbit
in } U \}.
\]
We refer to this as a local unstable manifold of $p$ relative to $U$. If $p$
is hyperbolic, then the unstable manifold theorem gives that there exists an
open neighborhood $U$ of $p$ so that $W_{\mathrm{loc}}^{u}(p, U)$ is a $k - m$
dimensional embedded disk, tangent at $p$ to the unstable eigenspace of
$Df(p)$. If $f$ is analytic the embedding is analytic. If $f$ is invertible
then the unstable manifold of $p$ under $f$ can be seen as the stable manifold
of $p$ under $f^{-1}$. In the case that $f$ is invertible, the unstable
manifold of $p$ is
\begin{align*}
W^{u}(p)  & = \left\{ w \in\mathbb{R}^{n} : \lim_{n\to\infty} f^{-n}(w) = p
\right\} \\
& = \bigcup_{n=0}^{\infty}f^{n}\left[ W_{\mathrm{loc}}^{u}(p,U)\right] .
\end{align*}

\subsection{Heteroclinic Arcs}

\label{sec:heteroArcs}

\label{sec:primary} Suppose that $p_{1}, p_{2}$ are hyperbolic fixed points of
an invertible map $f \colon\mathbb{R}^{k} \to\mathbb{R}^{k}$, and that the
invariant manifolds $W^{u}(p_{1})$ and $W^{s}(p_{2})$ are of dimension $u_{1}$
and $s_{2}$ respectively. Assume that $u_{1} + s_{2} = k+1$. If $q$ is a point
in the transverse intersection of $W^{u}(p_{1})$ and $W^{s}(p_{2})$ it then
follows that there is an arc $\gamma\colon[-a, a] \to\mathbb{R}^{k}$ having
that $\gamma(0) = q$ and that
\[
\gamma(s) \subset W^{u}(p_{1}) \cap W^{s}(p_{2}), \quad\quad\quad\text{for
all} \quad s \in[-a, a].
\]
Moreover this intersection is transverse, hence $\gamma$ is as smooth as $f$.
We refer to $\gamma$ as a \textit{heteroclinic arc}.

Since the point $\gamma(s)$, $s\in\lbrack-a,a]$ is heteroclinic from $p_{1}$
to $p_{2}$, it follows that the set
\[
\mathcal{S}(\gamma)=\bigcup_{n\in\mathbb{Z}}f^{n}(\gamma)\subset W^{u}%
(p_{1})\cap W^{s}(p_{2}),
\]
is invariant. Moreover the entire set is heteroclinic from $p_{1}$ to $p_{2}$
so that
\[
\overline{\mathcal{S}\left(  \gamma\right)  }=\mathcal{S}(\gamma)\cup
\{p_{1}\}\cup\{p_{2}\},
\]
i.e. $\mathcal{S}(\gamma)$ accumulates only at the fixed points. Then
$\overline{\mathcal{S}(\gamma)}$ is a compact invariant set. We refer to
$\overline{\mathcal{S}(\gamma)}$ as the heteroclinic invariant set generated
by $\gamma$.

Suppose that $\gamma$ can be continued to a longer curve $\tilde{\gamma}$,
i.e. suppose that there is $\tilde{\gamma}\colon\lbrack-b,b]\rightarrow
\mathbb{R}^{k}$ with $a<b$ and $\gamma=\tilde{\gamma}|_{[-a,a]}$. Then
\[
\overline{\mathcal{S}\left(  \gamma\right)  }\subset\overline{\mathcal{S}%
\left(  \tilde{\gamma}\right)  }.
\]
We are interested in the largest compact invariant set so obtained, and single
out two cases of particular interest. \medskip

\begin{definition}[fundamental heteroclinic loop] \label{def:loop}{\em
Suppose that $\gamma(-a)=\gamma(a)$, i.e. $\gamma$ is a closed loop.
Then the heteroclinic invariant set $\mathcal{S}\left(  \gamma\right)  $ is
the union of countably many closed loops accumulating at $p_{1}$ and $p_{2}$.
If each point $q\in\gamma$ is a point of transverse intersection of
$W^{u}(p_{1})$ and $W^{s}(p_{2})$ then the closed loop $\gamma$ has no self
intersections. Moreover the lack of self intersections propagates under
forward and backward iteration as $f$ is a diffeomorphism. In this case we
refer to $\gamma$ as a \textit{fundamental heteroclinic loop}. \smallskip
}
\end{definition}

An example of a heteroclinic loop connection for the Lomel\'{\i} map is given in
Figure \ref{fig:LomeliMap2}. This connection is generated by a single loop,
propagated under forward and backward iterations of $f$. \bigskip

\begin{definition}[$m$-fold fundamental heteroclinic arc]
{\em If $\gamma(a)=f^{m}(\gamma(-a))$ for some $m\geq1$ and
$f^{i}(\gamma\left(  -a\right)  )\notin\gamma$ for $i<m$, then the $m$-th
iterate of the arc $\gamma$ is a continuation of $\gamma$. (If $f^{m}%
(\gamma(a))=\gamma(-a)$ then reparameterize.) We refer to $\gamma$ as an
$m$\textit{-fold fundamental heteroclinic arc.} Now
\[
\mathcal{S}_{m}\left(  \gamma\right)  :=\bigcup_{i\in\mathbb{Z}}f^{im}%
(\gamma)
\]
is itself an arc connecting $p_{1}$ and $p_{2}$, which we refer to as the
\emph{heteroclinic path generated by }$\gamma$.

For any $j\in\{ 1,\ldots,m-1\}$ the set $f^{j}(\mathcal{S}_{m}\left(
\gamma\right)  )$ is another heteroclinic path from $p_{1}$ to $p_{2}$. If
each point $q\in\gamma$ is a point of transverse intersection of $W^{u}%
(p_{1})$ and $W^{s}(p_{2})$ then for each $j\in\{ 1,\ldots,m-1\}$, the sets
$f^{i}(\mathcal{S}_{m}\left(  \gamma\right)  ),$ $f^{i+j}(\mathcal{S}%
_{m}\left(  \gamma\right)  )$ are disjoint. We refer to
\[
\overline{\mathcal{S}\left(  \gamma\right)  }=\bigcup_{j=0}^{m-1}%
f^{j}(\overline{\mathcal{S}_{m}\left(  \gamma\right)  })=\bigcup
_{i\in\mathbb{Z}}f^{i}\left(  \gamma\right)  \cup\left\{  p_{1}\right\}
\cup\left\{  p_{2}\right\}  ,
\]
as an $m$\textit{-}\emph{fold heteroclinic branched manifold,} as
$\overline{\mathcal{S}\left(  \gamma\right)  }$ is composed of $m$ paths (or
branches).
}
\end{definition}

\medskip

An example of a heteroclinic branched manifold for the Lomel\'{\i} map is given
in Figure \ref{fig:LomeliMap1}. It consists of six paths, which are generated
by two distinct $3$-fold fundamental heteroclinic arcs.

In both the case of the heteroclinic arcs, and the case of heteroclinic loops, the compact invariant sets $\overline
{\mathcal{S}\left(  \gamma\right)  }$ are maximal with respect to $\gamma$, in
the sense that no larger invariant set can be obtained by continuation of the
arc $\gamma$. In either case we refer to $\gamma$ as a \textit{fundamental
heteroclinic arc}. Note that under the assumption that $\gamma$ arises as the
one dimensional transverse intersection of smooth manifolds, the
classification theorem for one dimensional manifolds gives that only these two
cases occur.

\subsection{Vortex bubbles and the Lomel\'{\i} map}

\label{sec:vortexBubbles}

\begin{figure}[h]
\label{fig:bubbleSketch}
\par
\begin{center}
\includegraphics[height = 4cm]{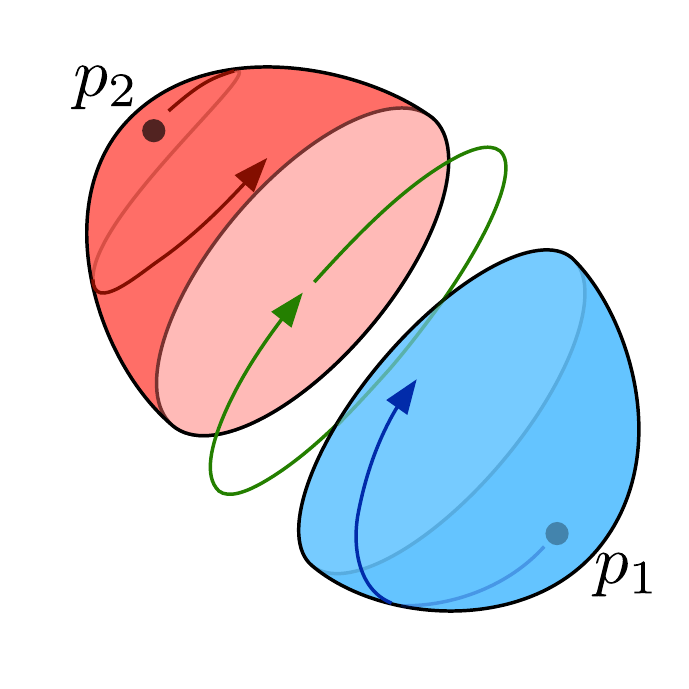}
\end{center}
\caption{Vortex bubble: in this sketch $p_{1}$ and $p_{2}$ are hyperbolic
fixed points with two dimensional unstable and two dimensional stable
manifolds respectively. The indicated unstable and stable eigenvalues occur in
complex conjugate pairs, giving the system a ``twist'' or circulation at the
fixed points. The circulation is sympathetic (clockwise or counterclockwise at
$p_{1}$ and $p_{2}$) and in the region between the fixed points there is a
``vortex''.}%
\end{figure}

In the sequel we restrict our attention to a particular dynamical
configuration known as a \textit{vortex bubble}. Heteroclinic arcs play an
important role in the study of vortex bubbles, and vortex bubbles in
$\mathbb{R}^{3}$ provide a non-trivial application, which can still be
completely visualized. Vortex bubbles appear in the fluid dynamics and plasma
physics literature as a model of turbulent circulation. See for example
\cite{MR2394421, MR1285950, MR1856972} and the references discussed therein.
At present we provide a brief qualitative sketch sufficient to our needs.

The main features of a vortex bubble are as follows. Consider $f
\colon\mathbb{R}^{3} \to\mathbb{R}^{3}$ a volume preserving diffeomorphism
(which could arise as a time one map of a volume preserving flow) having a
pair of distinct hyperbolic fixed points $p_{1}, p_{2} \in\mathbb{R}^{3}$. We
suppose that $p_{1}$ has two dimensional unstable manifold, and that $p_{2}$
has two dimensional stable manifold. Moreover we suppose that the unstable
eigenvalues at $p_{1}$ and the stable eigenvalues at $p_{2}$ occur in complex
conjugate pair, hence the linear dynamics at each fixed point is rotational.

We assume that the rotation is in the same direction at the fixed points, and
also that the curvature at $p_{1}$ and $p_{2}$ of the local unstable/stable
manifolds is such that the manifolds bend or ``cup'' toward one another.
Should the two dimensional global stable/unstable manifolds enclose a region
we say that a bubble (or resonance bubble) is formed. Under these conditions
it is not unusual that the circulation at the fixed points drives a
circulation throughout the bubble, in which case we say that there is a vortex
bubble. Inside the vortex bubble one may find invariant circles and tori, as
well as complicated chaotic motions.

The situation just described is sketched in Figure \ref{fig:bubbleSketch}. An
important global consideration is the intersection of the two dimensional
unstable and stable manifolds which, if transverse, gives rise to heteroclinic
arcs as discussed in Section \ref{sec:heteroArcs}.

One elementary mathematical model exhibiting vortex bubble dynamics is the
five parameter family of quadratic volume preserving maps
\begin{equation}
f(x,y,z)=\left(
\begin{array}
[c]{c}%
z+\alpha+\tau x+ax^{2}+bxy+cy^{2}\\
x\\
y
\end{array}
\right) , \label{eq:LomeliMap}%
\end{equation}
with $a+b+c=1$. We refer to this as the Lomel\'{\i} family, or simply as the
Lomel\'{\i} map. The map is a natural generalization of the H\'enon map from two
to three dimensions, and is the subject of a number of studies
\cite{MR2481277,  MR1704974, MR1617998, MR3079670}. In particular the
Lomel\'{\i} map provides a normal form for volume preserving quadratic
diffeomorphisms with quadratic inverse, and is a toy model for turbulent
fluid/plasma flow near a vortex \cite{MR1704974, MR3079670, MR2728178,
MR3068557}.

For typical parameters the map has two hyperbolic fixed points $p_{1}$ and
$p_{2}$ with stability as discussed above. The global embedding of the stable
and unstable manifolds for the system is illustrated in two specific instances
in Figures \ref{fig:LomeliMap1} and \ref{fig:LomeliMap2} of the Introduction.
These
computations suggest that both case 1 of heteroclinic arcs (in
this case) 3-fold, and case 2 of heteroclinic loops occur for the system as defined in Section
\ref{sec:heteroArcs} occur for the Lomel\'{\i} map. In the sequel we prove, by
a computer assisted argument, that this is indeed the case.


\section{Review of the parameterization method
for stable/unstable manifolds of fixed points}
\label{sec:parmMeth}

Let $f \colon \mathbb{C}^k \to \mathbb{C}^k$ be a smooth map
and suppose that 
$p \in \mathbb{C}^k$ is a hyperbolic fixed point as in 
Section \ref{sec:stableMan}.  Then the eigenvalues 
$\lambda_1, \ldots, \lambda_k \in \mathbb{C}$ for  
$Df(p)$ have
\[
|\lambda_j| \neq 1, 
\quad \quad \text{for all} \quad \quad
1 \leq j \leq k,
\]
i.e. none of the eigenvalues are on the unit circle.
Let $\lambda_1, \ldots, \lambda_s \in \mathbb{C}$
denote the stable eigenvalues of $Df(p)$, where 
$s \leq k$.  We order the stable eigenvalues so that 
\[
|\lambda_s| \leq \ldots  \leq | \lambda_1| < 1,
\]
and so that $\lambda_j$ is unstable when $s < j \leq k$.

For the sake of simplicity, suppose that $Df(p)$ is diagonalizable
and let $\xi_1, \ldots, \xi_k \in \mathbb{C}^k$ denote a choice of 
associated eigenvectors.  Then 
\[
Df(p) = Q \Lambda Q^{-1},
\] 
where 
\[
\Lambda = \left(
\begin{array}{ccc}
\lambda_1 & \ldots & 0 \\
\vdots & \ddots & \vdots \\
0 & \ldots & \lambda_k 
\end{array}
\right),
\]
is the $k \times k$ diagonal matrix of eigenvalues and 
\[
Q = [\xi_1, \ldots, \xi_k],
\]
is the $k \times k$ matrix whose $j$-th 
column is the eigenvector associated with $\lambda_j$.
We write 
\[
\Lambda_s = \left(
\begin{array}{ccc}
\lambda_1 & \ldots & 0 \\
\vdots & \ddots & \vdots \\
0 & \ldots & \lambda_s 
\end{array}
\right),
\]
to denote the $s \times s$ diagonal matrix of stable eigenvalues.

In the present work we assume that $f$ is analytic 
in a neighborhood of $p$.
Let  
\[
D_s = \{ \mathbf{z} = (z_1, \ldots, z_s) \in \mathbb{C}^s : |z_j| < 1 
\, \mbox{for each} \, 1 \leq j \leq s \},
\]
denote the $s$-dimensional unit poly-disk in $\mathbb{C}^s$.
The goal of the Parameterization Method is to find an analytic map 
$P \colon D_s \to \mathbb{C}^k$ having that 
\begin{equation} \label{eq:parmLinearConstraints}
P(0) = p, \quad \quad \quad \quad DP(0) = [\xi_1, \ldots, \xi_s],
\end{equation}
and 
\begin{equation}\label{eq:parmInvEq}
f[P(z_1, \ldots, z_s)] = P(\lambda_1 z_1, \ldots, \lambda_s z_s),
\end{equation}
for all $(z_1, \ldots, z_s) \in D_s$.  
Such a map $P$ parameterizes a local stable manifold for $f$
at $p$, as the following lemma makes precise.

\begin{lemma}\label{lem:Par-lem}
Suppose that $P$ is an analytic map satisfying the first 
order constraints in \eqref{eq:parmLinearConstraints} and solving 
Equation \eqref{eq:parmInvEq} in $D_s \subset \mathbb{C}^s$. 
Then $P$ is a chart map 
for a local stable manifold at $p$, i.e.  
\begin{enumerate}
\item for all $z \in P[D_s]$ the orbit of $z$ accumulates
at $p$,
\item $P[D_s]$ is tangent to the stable eigenspace at $p$, 
\item $P$ is one-to-one on $D_s$, i.e. $P$ is a chart map.
\end{enumerate} 
\end{lemma}

The proof of Lemma \ref{lem:Par-lem} follows along the same lines as Section 
$3.1$ of \cite{MR3079670}. Here we sketch the argument: Point $1.$ is seen by iterating the invariance equation \eqref{eq:parmInvEq},
considering the continuity of $P$ and the requirment that $P(0) = p$.  
 Point $2.$ follows directly from the first order
constraint on $DP(0)$ given in Equation \eqref{eq:parmLinearConstraints}, and 
point $3.$ is seen by applying the implicit function theorem 
to show that $P$ is one-to-one in a small neighborhood
of the origin in $D_s$, and then using that $f$ and $\Lambda_s$ 
are invertible maps to show that $P$ is one-to-one anywhere
that Equation \eqref{eq:parmInvEq} holds.

Since we seek $P$ analytic on a disk and satisfying first order constraints
its natural to look for a power series representation
\begin{equation}\label{eq:powerSeries}
P(z_1, \ldots, z_s) = \sum_{\alpha_1 = 0}^\infty \ldots \sum_{\alpha_s = 0}^\infty
p_{\alpha_1,  \ldots, \alpha_s} z_1^{\alpha_1} \ldots z_s^{\alpha_s}
= \sum_{|\alpha| = 0}^\infty p_\alpha z_\alpha, 
\end{equation}
where $\alpha = (\alpha_1, \ldots, \alpha_s) \in \mathbb{N}^s$ is an 
$s$-dimensional multi-index, 
\[
|\alpha| = \alpha_1 + \ldots + \alpha_s,
\] 
$p_{\alpha_1, \ldots, \alpha_s} = p_\alpha
\in \mathbb{C}^k$ for each $\alpha \in \mathbb{N}^s$, and 
$z^\alpha = z_1^{\alpha_1} \ldots z_s^{\alpha_s} \in \mathbb{C}$
for each $z \in D_s$, $\alpha \in \mathbb{N}^s$.
Imposing the first order constraints \eqref{eq:parmLinearConstraints}
gives
\[
p_{0, \ldots, 0} = 0, 
\quad \quad \mbox{and} \quad \quad 
p_{e_j} = \xi_j, 
\]
where for each $1 \leq j \leq s$ the multi-index $e_j$ is given by 
$e_j = (0, \ldots, 1, \ldots, 0)$, i.e. $e_j$ has a one in the $j$-th entry
and zeros elsewhere.
The Taylor coefficients $p_\alpha$ for $|\alpha| \geq 2$
can be determined by a power matching scheme.   
This procedure is illustrated by example in the next section.

\begin{remark}[Unstable manifold parameterization]\label{rem:unstableParm}
{\em
The considerations above apply to the unstable manifold of $f$ at $p$ by 
considering the stable manifold of the inverse map $f^{-1}$ at $p$.  
In fact the equation becomes 
\[
f^{-1}[P(z_1, \ldots, z_u)] = P(\sigma_1 z_1, \ldots, \sigma_u z_u),
\]
where $u$ is the number of stable eigenvalues of $D f^{-1}(p)$
and $\sigma_1, \ldots, \sigma_u$ denote these stable eigenvalues.  
Of course the stable eigenvalues for $Df^{-1}(p)$ are the reciprocals
of the unstable eigenvalues for $Df(p)$ so that $u = k - s$.  
Applying $f$ to both sides of the equation and pre-composing 
$P$ with $\sigma_1^{-1}, \ldots, \sigma_u^{-1}$ we obtain
\[
P(\sigma^{-1}_1 z_1, \ldots, \sigma^{-1}_u z_u) = f[P(z_1, \ldots, z_u)],
\]
with $\sigma^{-1}_i = \lambda_i$ the unstable 
eigenvalues of $Df(p)$.  This shows that chart maps 
for the stable and unstable manifolds satisfy the same Equation 
\eqref{eq:parmInvEq}.  The difference is that in one 
case we conjugate to the linear map given by 
the stable eigenvalues of $Df(p)$ and in 
the other case the linear map given by the unstable eigenvalues.  
}
\end{remark}

\subsection{Example: 2D manifolds associated 
with complex conjugate eigenvalues for the Lomel\'{\i} family}
In this section we write out in full detail how the  parameterization method works on the concrete example of the Lomel\'{\i} map. 
We include these formal computations for the sake of completeness.

Suppose that $p \in \mathbb{R}^3$ is a fixed point of the Lomel\'{\i} map
and that $Df(p)$ has a pair of stable complex conjugate eigenvalues 
$\lambda, \bar \lambda \in \mathbb{C}$, i.e. $|\lambda| = |\bar \lambda| < 1$.
Let $\xi, \bar \xi \in \mathbb{C}^3$ denote the complex conjugate 
eigenvectors.  Note that since the Lomel\'{\i} map is 
volume preserving, it is the case that the remaining 
eigenvalue is real and unstable.

Take $v, w$ in the unit disk in $\mathbb{C}$ and write
\[
P(v, w) = 
\left(
\begin{array}{c}
P_1(v,w) \\
P_2(v,w) \\
P_3(v,w)
\end{array}
\right)
= \sum_{k = 0}^\infty \sum_{l = 0}^\infty 
\left(
\begin{array}{c}
p_{k l}^1 \\
p_{k l}^2 \\
p_{k l}^3 
\end{array}
\right) v^k w^l.
\]
We see that 
\[
P(0,0) = 
\left(
\begin{array}{c}
p_{00}^1 \\
p_{00}^2 \\
p_{00}^3 
\end{array}
\right) = p,
\quad
\frac{\partial}{\partial v} P(0,0) = 
\left(
\begin{array}{c}
p_{10}^1 \\
p_{10}^2 \\
p_{10}^3 
\end{array}
\right) = \xi,
\quad
\frac{\partial}{\partial w} P(0,0) = 
\left(
\begin{array}{c}
p_{01}^1 \\
p_{01}^2 \\
p_{01}^3 
\end{array}
\right) = \bar \xi,
\]
by imposing the first order constraints of 
Equation \eqref{eq:parmLinearConstraints}. 
Plugging the unknown power series expansion of $P$ 
into the invariance Equation \eqref{eq:parmInvEq} gives 
\begin{align*}
& f[P(v,w)] = \\
& =\left(
\begin{array}{c}
P_3(v,w) + \alpha + \tau P_1(v,w) + 
a P_1(v,w)^2 + b P_1(v,w) P_2(v,w) + c P_2(v,w)^2 \\
P_1(v,w) \\
P_2(v,w)
\end{array}
\right) \\
&= 
{\sum_{k=0}^\infty \sum_{l=0}^\infty
\left(
\begin{array}{c}
p_{kl}^3 + \delta_{kl} \alpha + \tau p_{kl}^1 + 
\sum_{i = 0}^k \sum_{j = 0}^l \left(a p_{k-i l-j}^1 p_{ij}^1
+b p_{k-i l-j}^1 p_{ij}^2 + c  p_{k-i l-j}^2 p_{ij}^2\right) \\
p_{kl}^1 \\
p_{kl}^2
\end{array}
\right) v^k w^l,}
\end{align*}
on the left (where $\delta_{kl} = 0$ if $k=0$ and 
$l=0$ and $\delta_{kl} = 1$ otherwise), and 
\[
P(\lambda v, \bar \lambda w) = 
\sum_{k=0}^\infty \sum_{l = 0}^\infty \lambda^k \bar \lambda^l
\left(
\begin{array}{c}
p_{kl}^1 \\
p_{kl}^2 \\
p_{kl}^3 
\end{array}
\right)
v^k w^l,
\]
on the right.
Matching like powers leads to 
\begin{equation*}
 \left(
\begin{array}{c}
p_{kl}^3 + \delta_{kl} \alpha + \tau p_{kl}^1 + 
\sum_{i = 0}^k \sum_{j = 0}^l a p_{k-i l-j}^1 p_{ij}^1
+b p_{k-i l-j}^1 p_{ij}^2 + c  p_{k-i l-j}^2 p_{ij}^2 \\
p_{kl}^1 \\
p_{kl}^2
\end{array}
\right) 
= \lambda^k \bar \lambda^l
\left(
\begin{array}{c}
p_{kl}^1 \\
p_{kl}^2 \\
p_{kl}^3 
\end{array}
\right)
\end{equation*}
for all $k + l \geq 2$.
Extracting terms of order $kl$ and isolating them on the left hand side
leads to the linear \textit{homological} equations
\begin{equation}\label{eq:parmLinEqnsLomeli}
\left[
\begin{array}{ccc}
  \tau + 2 a + b - \lambda^k \bar \lambda^l &  b + 2c  &  1  \\
 1   & - \lambda^k \bar \lambda^l   & 0   \\
 0   & 1   &  - \lambda^k \bar \lambda^l
\end{array}
\right]
\left(
\begin{array}{c}
p_{kl}^1 \\
p_{kl}^2 \\
p_{kl}^3
\end{array}
\right) = s_{kl},
\end{equation}
where
\[
s_{kl} = 
\left(
\begin{array}{c}
  - \sum_{i = 0}^k \sum_{j = 0}^l \hat \delta_{ij}
  \left(a p_{k-i l-j}^1 p_{ij}^1
+b p_{k-i l-j}^1 p_{ij}^2 + c  p_{k-i l-j}^2 p_{ij}^2 \right)
   \\
   0 \\
   0 \\
\end{array}
\right)
\]
and the coefficient
\[
\hat \delta_{ij} = 
\begin{cases}
0  & i = 0  \text{ and } j = 0, \\
0  & i = k  \text{ and } j = l, \\
1 & \text{otherwise,}
\end{cases}
\]
accounts for the fact that terms of order $kl$
have been extracted from the sums.
In other words: the $s_{kl}$ depend only on terms $p_{ij}$ where
$i + j < k + l$.

The question arises: for what $k, l$ with $k + l \geq 2$
does the linear Equation \eqref{eq:parmLinEqnsLomeli}
have a unique solution?
Direct inspection of the formula for $Df(p)$ when $f$
is the Lomel\'{\i} map allows us to rewrite the homological equation
as
\begin{equation} \label{eq:homEq}
\left( Df(p) - \lambda^k \bar \lambda^l \mbox{Id} \right) p_{kl} = s_{kl}.
\end{equation}
Note that the matrix on the left hand side 
is characteristic for $Df(p)$.  In other words, this matrix 
is invertible as long as $\lambda^k \bar \lambda^l$
is not an eigenvalue of $Df(p)$.  
Since both $\lambda^k \bar \lambda^l = \lambda$
and $\lambda^k \bar \lambda^l = \bar \lambda$ are impossible for 
$k + l \geq 2$, and since the remaining eigenvalue is unstable,
we see that this matrix is invertible for all cases of concern.  
Then the Taylor coefficients of $P$ are formally well defined to 
all orders. Moreover we obtain a numerical 
algorithm by recursively solving the homological equations
to any desired order.
We write 
\[
P^N(v, w) = \sum_{n = 0}^N 
\sum_{k + l = n} p_{kl} v^k w^l,
\]
to denote the $N$-th order polynomial obtained
by solving the homological equations to order $N$.
We remark that by solving these homological equations
using interval arithmetic we obtain validated interval enclosures
of the true coefficients. 
This discussion goes through unchanged
if $\lambda, \bar \lambda \in \mathbb{C}$ are 
unstable rather than stable eigenvalues.

\begin{remark}
{\em
Note that the Lomel\'{\i} map is real analytic, and for 
real fixed points $p \in \mathbb{R}^3$ we are interested 
in the real image of $P$.  Using the fact that $Df(p)$ is 
a real matrix when $p$ is a real fixed point, we see that 
\[
\overline{
Df(p) - \lambda^k \bar \lambda^l  \mbox{Id}}
= Df(p) - \bar \lambda^k  \lambda^l \mbox{Id}, 
\] 
and similarly it is easy to check that $\overline{s_{kl}} 
= s_{lk}$ for all $k,l$.  Since we choose the first order 
coefficients so that 
\[
\overline{p_{10}} = \overline{\xi} = p_{01},
\]
it follows that the solutions of all homological equations
inherit this property, i.e. that 
\[
\overline{p_{kl}} = p_{lk}
\]
for all $k + l \geq 2$.  

Choosing complex conjugate variables 
\[
v = s + it 
\quad \quad \mbox{and} \quad \quad w = s - it,
\]
with $s, t$ real we see that the map $\hat P \colon B \to \mathbb{R}^3$
given by 
\[
\hat P(s, t) = P(s + it, s-it),
\]
is real valued.  Moreover the constraint that $v, w$ lie in the unit 
Poly-disk in $\mathbb{C}^2$ imposes that 
\[
B_2 = \left\{(s, t) \in \mathbb{R}^2 : \sqrt{s^2 + t^2} < 1\right\},
\]
is the natural domain for $\hat P$.  
These remarks show also that the
truncated polynomial approximation $P^N$ has complex 
conjugate coefficients, and hence a real image when we make
use the complex conjugate variables, as long as we include
each coefficients $p_{\alpha}$ with $2 \leq |\alpha| \leq N$ in our
approximation.
}
\end{remark}
\begin{remark}[Uniqueness and decay rate of the coefficients]
{\em
The discussion of the homological equation above 
shows that the Taylor coefficients $p_{kl}$ are unique
up to the choice of the length of $\xi$.  
In fact it can be shown that if $\hat{p}_{kl}$ are 
the coefficients associated with the scaling 
$|\xi| = |\bar \xi| = 1$,
and $p_{kl}$ are the coefficients associated with the 
scaling $|\xi| = |\bar \xi| = \tau$, then we have the relationship
\[
p_{kl} = \tau^{k + l} \hat{p}_{kl}.
\]
A simple proof of this fact is found for example 
in \cite{maxManifolds}. 
Then the length of $\xi$ adjusts the decay rate of the 
power series coefficients.
This freedom is exploited in order to stabilize numerical computations.
}
\end{remark}
\begin{remark}[A-posteriori error]
{\em
The computations above are purely formal, and in practice 
we would like to measure the accuracy of the approximation
on a fixed domain.  In order to make such a 
measurement  we exploit that the lengths of the eigenvectors
tune the decay rates of the power series coefficient, so that 
we can always fix the domain of $P^N$ to be the unit disk.  
Equation \eqref{eq:parmInvEq} then suggests we define the
a-posteriori error functional
\[
\epsilon_N = \sup_{|v|, |w| < 1} 
\| f[P^N(v,w)] - P^N(\lambda v, \bar \lambda w) \|_{\mathbb{C}^3}.
\]
This quantity provides a heuristic indicator of the quality of the 
approximation $P^N$ on the unit disk.  
Of course small defects do not necessarily imply small 
truncation errors. In Section \ref{sec:aPosParm} we discuss a 
method which makes this heuristic indicator precise.

Note that if $N$ is fixed then $\epsilon_N$ is a function of the length of 
$\xi$ only.  Then by varying the length we can make $\epsilon_N$ 
as small as we wish (up to machine errors).  To see this we simply 
note that $f \circ P^N$ is exactly equal to $P^N \circ \Lambda_s$ to
zero and first order, so that the function $f \circ P^N - P^N \circ \Lambda_s$
is zero to second order and has higher order coeffieicnts decaying as
fast as we wish.  Once a desired error is achieved fixed 
we increase $N$ and repeat the procedure.  In this way one
can optimize the size of the image of $P^N$ relative to a fixed 
desired error tolerance.  Again we refer the interested
reader to \cite{maxManifolds}, where such algorithms are 
discussed in more detail.
}
\end{remark}
\begin{remark}[Generalizations]
{\em
The computations sketched above succeed in much  
greater generality.  
For example when we study an $m$ dimensional 
manifold of an analytic mapping 
$f \colon \mathbb{C}^k \to \mathbb{C}^k$, then 
looking for a parameterization of the form 
\[
P(z) = \sum_{|\alpha| = 0}^\infty p_\alpha z^\alpha,
\]
leads to a homological equation of the form
\[
\left[ Df(p) - \lambda_1^{\alpha_1} \ldots
\lambda_s^{\alpha_s} \mbox{Id} \right] p_\alpha = s_\alpha,
\]
for the coefficients with $|\alpha| \geq 2$.  Here again
$s_\alpha$ depends only on terms of order lower 
than $|\alpha|$, and the form of $s_\alpha$ is determined
by the nonlinearity of $f$.
Note that the general case leads to the non-resonance conditions
\[
 \lambda_1^{\alpha_1} \ldots
\lambda_s^{\alpha_s} \neq \lambda_j 
\quad \quad \quad \mbox{for} \quad 1 \leq j \leq s.
\]
Inspection of these conditions shows that the
non-resonance conditions hold generically, i.e. they reduce to a 
finite collection of constraint equations. 
Moreover, should a resonance occur it is still possible for the Parameterization
Method to succeed.  However when there is a resonance,
rather than conjugating to the 
linear map generated by the stable eigenvalues,
it is necessary to conjugate the 
dynamics on the manifold to a certain 
polynomial which ``kills'' the resonant terms. The general
development of the Parameterization Method 
for stable/unstable manifolds of non-resonant fixed points
is in \cite{MR1976079, MR1976080, MR2177465}.
Validated numerical methods for the resonant case 
as well as the non-diagonalizable case are 
developed in \cite{parmChristian}.
}
\end{remark}

\subsection{A-posteriori validation and computer assisted error bounds}
\label{sec:aPosParm}
We say that an analytic function 
$h \colon D_s \to \mathbb{C}^k$ is an analytic $N$-tail 
if the Taylor coefficients of $h$ are zero to $N$-th order, 
i.e. if 
\[
h(z) = \sum_{|\alpha| = 0}^\infty h_\alpha z^\alpha,
\]
and 
\[
h_\alpha = 0, 
\quad \quad \quad \mbox{for } 0 \leq |\alpha| \leq N.
\]
Such functions are used to represent truncation errors in 
power series methods.
For $p \in \mathbb{C}^k$ and $R > 0$ let
\[
D_k (p, R) = \{z \in \mathbb{C}^k : \| z - p\| < R \},
\]
We make the following assumptions.

\medskip

\noindent \textbf{A1:} Assume that 
$f \colon D_k(p, R) \subset \mathbb{C}^k \to \mathbb{C}^k$
is analytic and that $p \in \mathbb{C}^k$ has $f(p) = p$.

\medskip

\noindent \textbf{A2:} Assume that $Df(p)$ is nonsingular, diagonalizable, and 
hyperbolic.  Let $\{\lambda_1, \ldots, \lambda_{s}\}$
denote the stable eigenvalues, $\{\xi_1, \ldots, \xi_s\}$ denote a 
choice of corresponding eigenvectors, $\Lambda_s$ denote the 
$s \times s$ diagonal matrix of stable eigenvalues, and 
$A_s = [\xi_1, \ldots, \xi_s]$ denote the $k \times s$ matrix 
whose columns are the stable eigenvectors.

\medskip

\noindent \textbf{A3:}  Assume that $P^N \colon \mathbb{C}^s \to \mathbb{C}^k$
is an $N$-th order polynomial, with $N \geq 2$.  Assume that $P^N$ is an 
exact formal solution of the equation 
\[
f \circ P^N = P^N \circ \Lambda_s,
\]
to $N$-th order, that is, we assume that the power series of the right hand 
side is equal to the power series of the left hand side exactly up to $N$-th order.

\bigskip

The following definition collects some constants which are critical in the 
a-posteriori validation theorem to follow.

\begin{definition}[Validation values for the Stable Manifold]
{\em
Let $f \colon \mathbb{C}^k \to \mathbb{C}^k$ and 
$P^N \colon \mathbb{C}^s \to \mathbb{C}^k$ be as in assumptions 
$A1$- $A3$. A collection of positive constants $\epsilon_{\mathrm{tol}}$, 
$R$, $R'$, $\mu^*$, $K_1$ and $K_2$ are called 
\textit{validations values} for $P^N$ if 
\begin{equation}\label{eq:valVal1}
\sup_{z \in D_s} 
\| f [P^N(z)] - P^N (\Lambda_s z) \|_{\mathbb{C}^k} \leq \epsilon_{\mathrm{tol}},
\end{equation}
\begin{equation} \label{eq:valVal2}
\sup_{z \in D_s} \| P^N(z) - p \|_{\mathbb{C}^k} \leq R' < R,
\end{equation}
\begin{equation} \label{eq:valVal3}
0 < \max_{1 \leq j \leq s} |\lambda_j| \leq \mu^* < 1,
\end{equation}
\begin{equation}\label{eq:valVal4}
\sup_{z \in D_s} \| [Df]^{-1}(P^N(z))\| \leq K_1, 
\end{equation}
\begin{equation}\label{eq:valVal5}
\max_{\stackrel{\beta \in \mathbb{N}^m}{|\beta| = 2}} 
\max_{1 \leq j \leq k} \sup_{  q \in D_k(p, R)} \| \partial^\beta
f_j(q)\|_{\mathbb{C}} \leq K_2.
\end{equation}
}
\end{definition}

Some explanation of the meaning of these constants is
appropriate.  In Equation \eqref{eq:valVal1}, we see that 
$\epsilon_{\mathrm{tol}}$ measures the 
defect associated with the approximation $P^N$ on 
$D_s$.  The requirement that $R' < R$ in Equation  
\eqref{eq:valVal2} guarantees that the image of the approximate
parameterization $P^N$
is contained in the disk $D_k(p, R')$, i.e. strictly 
interior to the disk $D_k(p, R)$ on which we have control over 
derivatives.  The image of the true parameterization $P = P^N + H$
will live in the larger disk  $D_k(p, R)$.
Note that there is no assumption that either $R'$ or $R$ are
small.  Equation \eqref{eq:valVal3} postulates a uniform 
bound on the absolute values of the stable eigenvalues, while Equation 
\eqref{eq:valVal4} requires a uniform bound on the inverse 
of the Jacobian derivative of $f$ holding over
all of $D_s$.  Finally Equation \eqref{eq:valVal5} requires a 
uniform bound on the second derivatives of $f$ which is 
valid over all of $D_k(p, R)$.
Interval arithmetic computation of validation values
is discussed in \cite{MR3068557}, and such computations are
 implemented for the Lomel\'{\i} map in the same reference.

The following theorem is the main result of 
\cite{MR3068557}.  Its proof is found in the same reference.

\begin{theorem}[A-posteriori error bounds] \label{thm:validateManifolds}
Given assumptions $A1- A3$, suppose that 
$\epsilon_{\mathrm{tol}}$, $R$, $R'$, $\mu^*$, $K_1$, and $K_2$ are
validation values for $P^N$.  Let
\[
N_f = \# \{ \beta \in \mathbb{N}^k : |\beta| = 2 \quad \mbox{and} \quad
\partial^\beta f_j(q) \neq 0 \quad \mbox{for} \quad 
q \in D_k(p,R), 1 \leq j \leq k \},
\]
count the number of not identically zero partial derivatives of $f$, 
and suppose that $N \in \mathbb{N}$ and $\delta > 0$ satisfy the three
inequalities
\begin{equation}\label{eq:aPosEq1}
N+1 > \frac{-\ln(K_1)}{\ln(\mu^*)},
\end{equation}
\begin{equation}\label{eq:aPosEq2}
\delta < e^{-1} \min\left(
\frac{1 - K_1 (\mu^*)^{N+1}}{2 k \pi N_f K_1 K_2}, R - R'
\right),
\end{equation}
\begin{equation}\label{eq:aPosEq3}
\frac{2 K_1 \epsilon_{\mathrm{tol}}}{1 - K_1 (\mu^*)^{N+1}} < \delta.
\end{equation}
Then there is a unique analytic $N$-tail 
$h \colon D_m \subset \mathbb{C}^m \to \mathbb{C}^k$ having that 
\[
\sup_{z \in D_m} \| h(z)\|_{\mathbb{C}^k} \leq \delta,
\]
and that 
\[
P(z) = P^N(z) + h(z),
\]
is the exact solution of Equation \eqref{eq:parmInvEq}.  
\end{theorem}

The theorem provides explicit conditions, all checkable by finite
computations using interval arithmetic, sufficient to insure that there
is an analytic $N$-tail so that $P^N + h$ solve the invariance equation
for the Parameterization Method.  In this case $h$ is the truncation error 
on $D_s$ associated with stopping our Taylor approximation at $N$-th order. 
$\delta$ then provides an explicit bound on the truncation error.
Note also that the size of $\delta$ is related to the a-posteriori error 
$\epsilon_{\mathrm{tol}}$ times the quantity $K_1$, which in a sense 
measures how far from singular $Df$ is on the image of $P^N$. 
The theorem also gives an indication of how large the order of 
approximation $N$ must be taken.

In practice once validation values for $P^N$ are found one then 
checks that $N$ satisfies the condition given by 
Equation \eqref{eq:aPosEq1},  computes a bound $c_1$ on the quantity 
given in the right hand side of Equation \eqref{eq:aPosEq2}, 
computes a bound $c_2$ on the quantity given in the left hand 
side of Equation \eqref{eq:aPosEq3}, and then checks that 
$c_1 > c_2$.  Each of these computations and checks is done using 
interval arithmetic.  If this procedure succeeds then the theorem 
holds for any $\delta \in (c_2, c_1)$.  Of course in practice we then 
take $\delta$ as small as possible, i.e. very close to $c_2$.
Again these matters are discussed in more detail in \cite{MR3068557}.

\begin{remark}[Cauchy bounds on the derivative]\label{rem:cauchyBounds}
{\em 
Since the truncation error $h$ is an analytic function bounded by $\delta$
on the fixed disk $D_s$, we can bound derivatives of $h$ on any smaller 
disk using classical estimates of complex analysis.  
Indeed, let $f \colon D_m(\nu) \to \mathbb{C}^k$ be an analytic function with 
\[
\sup_{z \in D_m(\nu)} \| f(z)\| \leq M.
\]
Then for any $0 < \sigma \leq 1$ the Cauchy bounds
\[
\sup_{z \in D_m(\nu e^{-\sigma})}
 \left\| \frac{\partial}{\partial z_j} f(z)\right\| \leq
 \frac{2 \pi}{\nu \sigma} M,
\]
translate bounds on the size of a function into bounds on the 
size of its derivative on a strictly smaller disk. 
An elementary proof of the Cauchy bounds can be found for 
example in \cite{MR3068557}.
Repeated application
of the Cauchy bounds leads to estimates of $j$-th order derivatives 
inverse proportional to $(\nu \sigma)^j$.

Now suppose that $P^N$ and $h, P \colon D_s \to \mathbb{C}^k$ are 
as in Theorem  \ref{thm:validateManifolds}, 
so that $P(z) = P^N(z) + h(z)$ parameterizes a local stable 
manifold for $p$, and $\| h \| \leq \delta$ for 
$z \in D_s$.  Then for $0 < \sigma \leq 1$, 
$z \in D_s$, and $1 \leq j \leq s$ we have that 
\[
\frac{\partial}{\partial z_j} P(z) = \frac{\partial}{\partial z_j}P^N(z)
+ \frac{\partial}{\partial z_j} h(z).
\]
In practice the first term on the right is computed explicitly 
as the partial derivative of the known polynomial $P^N$,
while the Cauchy bounds applied to the second term on 
the right yield that
\[
\sup_{z \in D_s(e^{-\sigma})}
 \left\| \frac{\partial}{\partial z_j} h(z)\right\| \leq
 \frac{2 \pi}{\sigma} \delta.
\]
This decomposition is used
in order to control derivatives of the truncation errors
of the Parameterizations in the heteroclinic arc computations
to follow.  We just have to make sure that our heteroclinic
arcs are contained in the interior of the domain $D_s$.
}
\end{remark}

\subsection{Example stable/unstable manifold computations}

\begin{figure}[h] 
\begin{center}
\scalebox{0.35}{\includegraphics{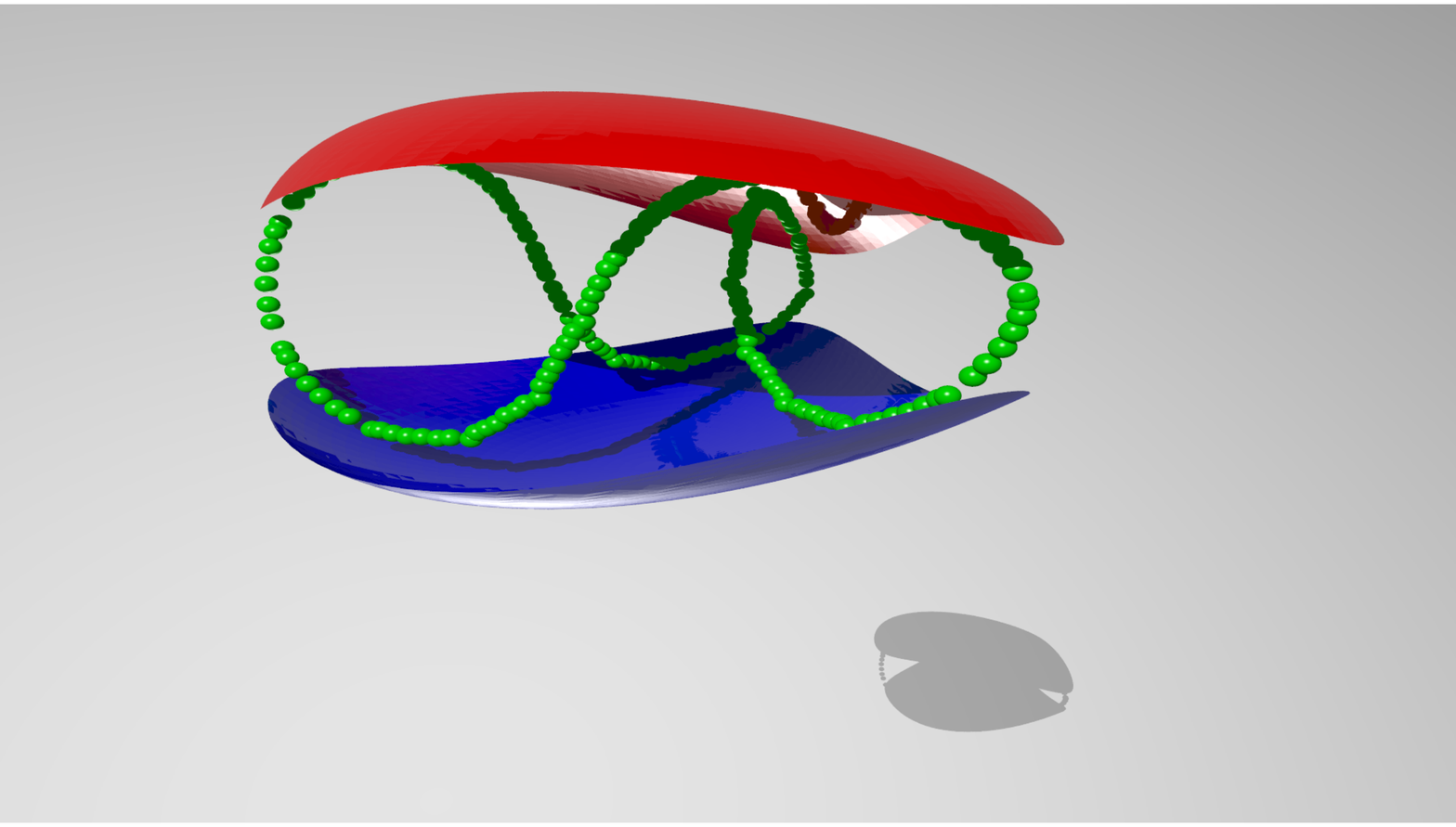}}
\end{center}
\caption{
Local stable/unstable manifolds for the Lomel\'{\i} map: 
images of polynomial parameterizations for the system with 
parameter values $a = 0.44$, $b = 0.21$,
$c = 0.35$, $\alpha = -0.25$, and $\tau = -0.3$, i.e. the 
same parameters used to generate Figure \ref{fig:LomeliMap1}.  
Local unstable manifold shown in blue and local 
stable manifold in red.
}\label{fig:bubbleInvArcs}
\end{figure}

\begin{figure}[h] 
\begin{center}
\scalebox{0.35}{\includegraphics{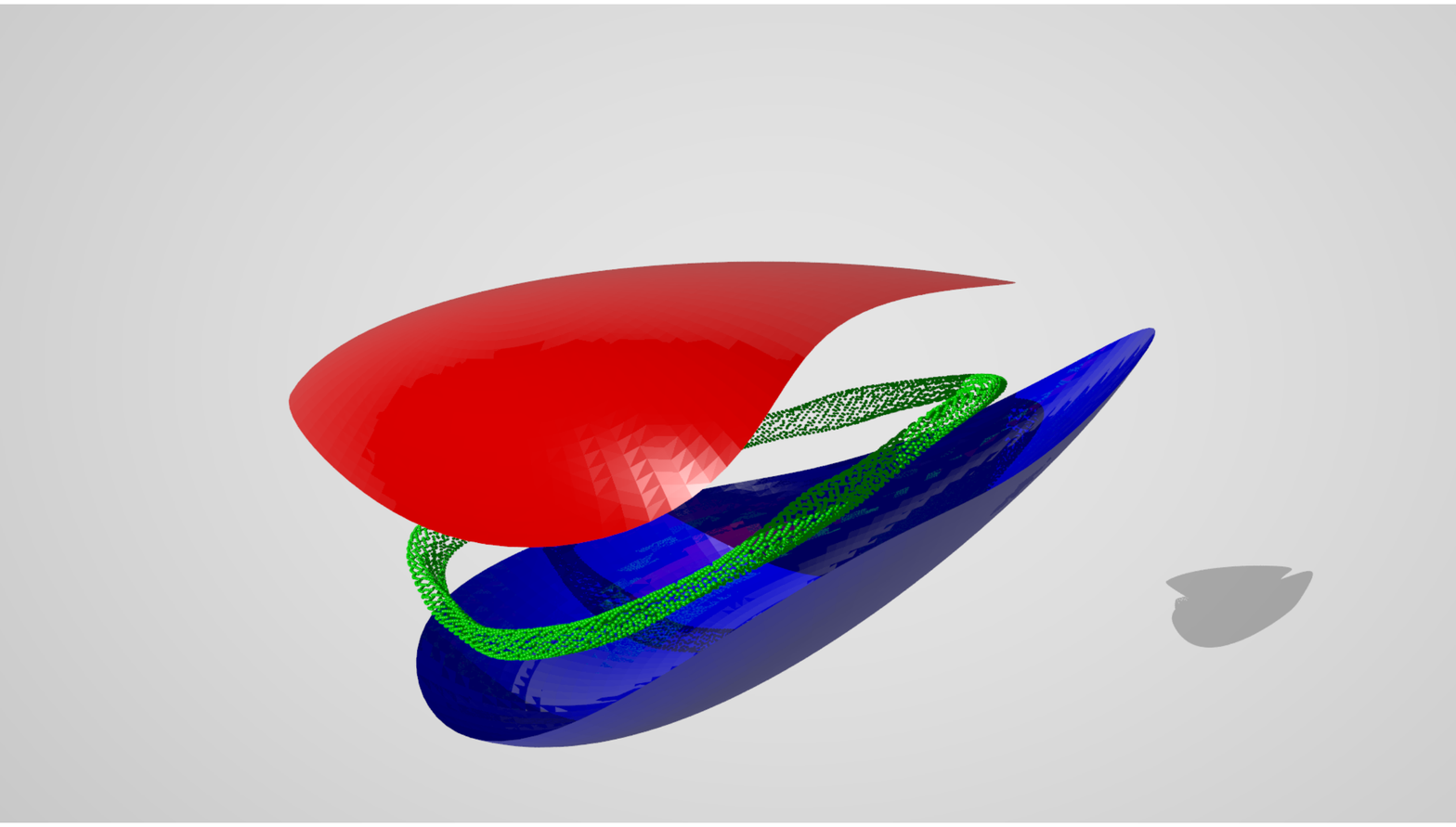}}
\end{center}
\caption{
Local stable/unstable manifolds for the Lomel\'{\i} map: 
images of polynomial parameterizations for the system with 
parameter values $a = 0.5$, $b = -0.5$,
$c = 1$, $\alpha = -0.08999$, and $\tau =  0.8$, i.e. the 
same parameters used to generate Figure \ref{fig:LomeliMap2}.  
Local unstable manifold shown in blue and local 
stable manifold in red.
} \label{fig:bubbleTorus}
\end{figure}

Figures \ref{fig:bubbleInvArcs} and \ref{fig:bubbleTorus}
illustrate the results of several computations for the Lomel\'{\i} map
utilizing the Parameterization Method.  Namely we compute
polynomial approximations of the manifolds to order $N = 45$
for the two dimensional local unstable and stable manifolds
at $p_1$ and $p_2$ respectively.  The figures result from computing
a triangular mesh in the domains of the parameterizations, and lifting
the mesh to the phase space using the polynomial chart maps.  
Moreover, the global stable/unstable manifolds shown in 
 Figure \ref{fig:LomeliMap1} are obtained from 
 the local manifolds shown in Figure \ref{fig:bubbleInvArcs} 
 after $30$ iterates.
Similarly, the global stable/unstable manifolds shown in 
 Figure \ref{fig:LomeliMap2} are obtained from 
 the local manifolds shown in Figure \ref{fig:bubbleTorus} 
 after $10$ iterates.

By verifying the hypotheses of Theorem \ref{thm:validateManifolds}
we obtain a-posteriori error bounds on the truncation error 
associated with the parameterizations.  In all cases the supremum 
norm errors are confirmed to be smaller than $10^{-9}$.
These validated parameterizations are used in the sequel in order to 
study the heteroclinic arcs for the bubbles.

\begin{remark}[Software and performance] \label{rem:parmComps}{\em
The parameterizations of all local stable and unstable manifolds used in the
present work are computed and validated using the INTLAB library for interval 
arithmetic running under MatLab \cite{intlabCitation}.  Performance and 
implementation of the validated computations is discussed in detail in
\cite{MR3068557}.
}
\end{remark}

\begin{remark}[Dynamics inside the bubble]\label{rem:bubbleInnerDynamics}
{\em
For the reader interested in the vortex bubble dynamics of the 
Lomel\'{\i} map we have included some additional dynamical 
information in Figures \ref{fig:bubbleInvArcs} and \ref{fig:bubbleTorus}.
In addition to plotting the parameterized local stable/unstable manifolds
we also considered a $500 \times 500 \times 500$ box of initial 
conditions in the ``bubble region''.  For the parameter values studied
in Figure \ref{fig:bubbleInvArcs} we find that a typical 
orbit escapes the region (and diverges to infinity).  
The green trajectory in Figure \ref{fig:bubbleInvArcs} was found by considering 
only orbits which stay in the bubble region for more than 500 
iterates.  We conjecture that there is an unstable invariant 
circle near this green orbit. 
 
For the parameter values studied in Figure \ref{fig:bubbleTorus}
typical orbits in the ``bubble region'' are invariant 
tori, chaotic orbits, or orbits which escape the region all together.
For example we plot as a 
green set in Figure  \ref{fig:bubbleTorus} an orbit which appears to lie
on an invariant torus. 
}  
\end{remark}

\section{Intersections of stable/unstable manifolds}

\label{sec:intersections}

In this section we discuss how to establish intersections of stable/unstable
manifolds of hyperbolic fixed points.

Let $f:\mathbb{R}^{k}\rightarrow\mathbb{R}^{k}$ be an invertible map and
$p_{1},p_{2}$ its hyperbolic fixed points with associated stable manifolds
$W^{s}(p_{i})$ and unstable manifolds $W^{u}(p_{i}),$ for $i=1,2.$ We assume
that $W^{u}(p_{1})$ is of dimension $u_{1}$ and that $W^{s}(p_{2})$ is of
dimension $s_{2}$, with $m=u_{1}+s_{2}>k$. Our objective is to investigate the
intersection of $W^{u}(p_{1})$ with $W^{s}(p_{2}).$ We shall formulate
conditions which ensure that they (locally) intersect transversally along an
$m-k$ manifold in $\mathbb{R}^{k}$.

\begin{remark}
In the setting of the Lomel\'{\i} map, we will have $k=3$, $u_{1}=s_{2}=2$ and
so the manifolds will intersect along $u_{1}+u_{2}-k=1$ dimensional curves. We
write our method in the more general context, to emphasize that it is
applicable also in higher dimensions.
\end{remark}

We assume that the manifolds $W^{u}(p_{1})$ and $W^{s}(p_{2})$ are
parameterized by
\[
P_{1}:B_{u_{1}}\rightarrow\mathbb{R}^{k},\qquad P_{2}:B_{s_{2}}\rightarrow
\mathbb{R}^{k},
\]
where $B_{u_{1}}$ and $B_{s_{2}}$ are used to denote balls centered at zero in
$\mathbb{R}^{u_{1}}$ and $\mathbb{R}^{s_{2}}$, respectively. We shall write
$\theta$ for coordinates on $\mathbb{R}^{u_{1}}$ and $\phi$ for coordinates on
$\mathbb{R}^{s_{2}}$.

Let
\begin{align}
F  &  :B_{u_{1}}\times B_{s_{2}}\rightarrow\mathbb{R}^{k},\nonumber\\
F(\theta,\phi)  &  =f^{l_{1}}(P_{1}(\theta))-f^{-l_{2}}\left(  P_{2}%
(\phi)\right)  , \label{eq:F-def}%
\end{align}
for some $l_{1},l_{2}\in\mathbb{N}$. We shall look for points $p^{\ast}\in
B_{u}\times B_{s}$ for which we will have
\begin{equation}
F(p^{\ast})=0. \label{eq:Fp-star-zero-intro}%
\end{equation}

A point $p^{\ast}$ satisfying (\ref{eq:Fp-star-zero-intro}) gives a point of
intersection of $W^{u}\left(  p_{1}\right)  $ and $W^{s}(p_{2})$ in the phase
space as $P_{1}\left(  \pi_{\theta}p^{\ast}\right)  $, or $P_{2}\left(
\pi_{\phi}p^{\ast}\right)  $. The two points come from the same homoclinic
orbit, i.e.
\[
f^{l_{1}+l_{2}}\left(  P_{1}\left(  \pi_{\theta}p^{\ast}\right)  \right)
=P_{2}\left(  \pi_{\phi}p^{\ast}\right)  .
\]

We see that finding intersections of $W^{u}\left(  p_{1}\right)  $ and
$W^{s}(p_{2})$ reduces to finding zeros of $F$. In section
\ref{sec:inter-general} we address this problem in general context, and then
apply the method to the Lomel\'{\i} map in section \ref{sec:Lom-application}.

\subsection{General setup\label{sec:inter-general}}

Let us consider a function
\[
F:\mathbb{R}^{m}\rightarrow\mathbb{R}^{k},
\]
where $m>k$. In this section we present an interval Newton type method for
establishing estimates on the set
\begin{equation}
\Sigma_{0}:=\left\{  F=0\right\}  . \label{eq:Fzero}%
\end{equation}
If $F$ is $C^{1}$ then we can expect $\Sigma_{0}$ to be a $C^{1}$ manifold of
dimension $m-k$. Our method will work in such setting.

Consider $\mathrm{x}\in\mathbb{R}^{m-k}$ and define a function $F_{\mathrm{x}%
}:\mathbb{R}^{k}\rightarrow\mathbb{R}^{k}$ as
\begin{equation}
F_{\mathrm{x}}\left(  \mathrm{y}\right)  :=F\left(  \mathrm{x},\mathrm{y}%
\right)  . \label{eq:F-kappa-def}%
\end{equation}
For $X\subset\mathbb{R}^{m-k}$ and $Y\subset\mathbb{R}^{k}$, by $DF_{X}\left(
Y\right)  $ we denote the family of matrixes
\[
DF_{X}\left(  Y\right)  =\left\{  D\left(  F_{\mathrm{x}}\right)  \left(
\mathrm{y}\right)  :\mathrm{x}\in X,\mathrm{y}\in Y\right\}  .
\]
Bounds on (\ref{eq:Fzero}) can be obtained by using the interval Newton
method. Below theorem is a well known modification (see for instance \cite[p.
376]{MR2652784}) of the method, that includes a parameter.

\begin{theorem}
\label{thm:Newton-enclosure-simple}Let $X=\Pi_{i=1}^{m-k}\left[  a_{i}%
,b_{i}\right]  \subset\mathbb{R}^{m-k}$ and $Y=\Pi_{i=1}^{k}\left[
c_{i},d_{i}\right]  \subset\mathbb{R}^{k}$. Consider $\mathrm{y}_{0}%
\in\mathrm{int}Y$ and%
\[
N\left(  \mathrm{y}_{0},X,Y\right)  =\mathrm{y}_{0}-\left[  DF_{X}\left(
Y\right)  \right]  ^{-1}\left[  F_{X}\left(  \mathrm{y}_{0}\right)  \right]
.
\]
If
\begin{equation}
N\left(  \mathrm{y}_{0},X,Y\right)  \subset\mathrm{int}Y,
\label{eq:Newton-direct}%
\end{equation}
then there exists a function $q:X\rightarrow Y$ such that $F\left(
\mathrm{x},q\left(  \mathrm{x}\right)  \right)  =0.$
\end{theorem}

\begin{remark}
By the implicit function theorem, $q\left(  \mathrm{x}\right)  $ is as smooth
as $F$.
\end{remark}

\begin{remark}
\label{rem:zeros-pts} If we choose $X$ to be a single point $X=\{\mathrm{x}_{0}\}%
$, then we can use Theorem \ref{thm:Newton-enclosure-simple} to establish an
enclosure $\left\{  \mathrm{x}_{0}\right\}  \times Y$ of the point $\left(
\mathrm{x}_{0},q\left(  \mathrm{x}_{0}\right)  \right)  $, for which $F\left(
\mathrm{x}_{0},q\left(  \mathrm{x}_{0}\right)  \right)  =0.$
\end{remark}

\begin{remark}
In Theorem \ref{thm:Newton-enclosure-simple} we have fixed $\mathrm{x}$ to be
the first $k$ coordinates. We can also apply the method by fixing any other
$k$ of the $m$ coordinates.
\end{remark}

For $X,Y$ from Theorem \ref{thm:Newton-enclosure-simple}, the set $X\times Y$
is an enclosure of $\Sigma_{0}.$ The theorem establishes the smoothness of
$\Sigma_{0}$ and proves that it is a graph over the $X$ coordinate. This
approach to obtaining an enclosure is simple and direct, but has one major
flaw. The main issue is that the bound on $F_{X}\left(  \mathrm{y}_{0}\right)
=F\left(  X,\mathrm{y}_{0}\right)  $ might not be tight, and in such case the
application of the method would require a choice of very small $X$. This in
practice could result in needing a vast number of sets to fully enclose
$\Sigma_{0}$. A natural remedy for keeping the enclosure of $F_{X}\left(
\mathrm{y}_{0}\right)  $ in check would be a more careful choice of local
coordinates. This is what motivates our next approach.

Assume that $p^{\ast}$ is a point for which we have
\begin{equation}
F(p^{\ast})=0. \label{eq:Fp-star-zero}%
\end{equation}
We will consider a neighborhood of $p^{\ast}$, in which we want to locally
enclose $\Sigma_{0}.$ Let $A_{1}$ be a $m\times(m-k)$ matrix and let $A_{2}$
be a $m\times k$ matrix. We will be looking for points of the form%
\[
p=p(\mathrm{x},\mathrm{y}):=p^{\ast}+A_{1}\mathrm{x}+A_{2}\mathrm{y}%
\]
for which
\begin{equation}
F\left(  p\left(  \mathrm{x},\mathrm{y}\right)  \right)  =0.
\label{eq:Fp-kappa-q-zero}%
\end{equation}

We will first formulate an interval Newton-type theorem that will allow us to
establish bounds on $\mathrm{x},\mathrm{y}$ solving (\ref{eq:Fp-kappa-q-zero}%
). Later we will follow with comments on how $A_{1}$ and $A_{2}$ should be
chosen and why the proposed approach can provide better estimates than Theorem
\ref{thm:Newton-enclosure-simple}.

\begin{theorem}
\label{th:directed-Newton} Let $A_{1}$ be an $m\times(m-k)$ matrix and let
$A_{2}$ be a $m\times k$ matrix. Let $X=\Pi_{i=1}^{m-k}\left[  a_{i}%
,b_{i}\right]  ,$ $Y=\Pi_{i=1}^{k}\left[  c_{i},d_{i}\right]  ,$
$\mathrm{x}_{0}\in X,$ $\mathrm{y}_{0}\in Y$ and let us introduce the
following notations%
\begin{align}
F_{\mathrm{x}_{0},\mathrm{y}_{0}}  &  :=F\left(  p^{\ast}+A_{1}\mathrm{x}%
_{0}+A_{2}\mathrm{y}_{0}\right)  ,\nonumber\\
DF_{X,\mathrm{y}_{0}}  &  :=DF\left(  p^{\ast}+A_{1}X+A_{2}\mathrm{y}%
_{0}\right)  ,\nonumber\\
DF_{X,Y}  &  :=DF\left(  p^{\ast}+A_{1}X+A_{2}Y\right)  ,\nonumber\\
N\left(  \mathrm{x}_{0},\mathrm{y}_{0},A_{1},A_{2},X,Y\right)   &
:=\mathrm{y}_{0}-\left[  DF_{X,Y}A_{2}\right]  ^{-1}\left(  F_{\mathrm{x}%
_{0},\mathrm{y}_{0}}+\left[  DF_{X,\mathrm{y}_{0}}A_{1}\right]  \left[
X-\mathrm{x}_{0}\right]  \right)  . \label{eq:Newton-operator}%
\end{align}
If
\[
N\left(  \mathrm{x}_{0},\mathrm{y}_{0},A_{1},A_{2},X,Y\right)  \subset Y,
\]
then there exists a function $q:X\rightarrow Y,$ such that%
\begin{equation}
F\left(  p^{\ast}+A_{1}\mathrm{x}+A_{2}q(\mathrm{x})\right)  =0.
\label{eq:zero-with-q}%
\end{equation}
Moreover, $q$ is as smooth as $F$.
\end{theorem}

\begin{proof}
Let us introduce the following notation. Let $g_{\mathrm{x}}:Y\rightarrow
\mathbb{R}^{k}$ and $h:X\rightarrow\mathbb{R}^{k}$ be defined as%
\begin{align*}
g_{\mathrm{x}}\left(  \mathrm{y}\right)   &  =F\left(  p^{\ast}+A_{1}%
\mathrm{x}+A_{2}\mathrm{y}\right)  ,\\
h\left(  \mathrm{x}\right)   &  =F\left(  p^{\ast}+A_{1}\mathrm{x}%
+A_{2}\mathrm{y}_{0}\right)  .
\end{align*}
By the mean value theorem, for any $\mathrm{x}\in X$%
\begin{align*}
g_{\mathrm{x}}(\mathrm{y}_{0})  &  =h\left(  \mathrm{x}\right) \\
&  \in h\left(  \mathrm{x}_{0}\right)  +\left[  Dh\left(  X\right)  \right]
\left[  X-\mathrm{x}_{0}\right] \\
&  =F_{\mathrm{x}_{0},\mathrm{y}_{0}}+\left[  DF_{X,\mathrm{y}_{0}}%
A_{1}\right]  \left[  X-\mathrm{x}_{0}\right]  .
\end{align*}
Also for any $\mathrm{x}\in X$,%
\[
\left[  Dg_{\mathrm{x}}\left(  Y\right)  \right]  =\left[  DF\left(  p^{\ast
}+A_{1}\mathrm{x}+A_{2}Y\right)  A_{2}\right]  \subset\left[  DF_{X,Y}%
A_{2}\right]  .
\]
This means that for any $\mathrm{x}\in X$%
\[
\mathrm{y}_{0}-\left[  Dg_{\mathrm{x}}\left(  Y\right)  \right]  ^{-1}\left[
g_{\mathrm{x}}\left(  \mathrm{y}_{0}\right)  \right]  \subset N\left(
\mathrm{x}_{0},\mathrm{y}_{0},A_{1},A_{2},X,Y\right)  \subset Y,
\]
hence by the interval Newton theorem for every $\mathrm{x}\in X$ we have
$q\left(  \mathrm{x}\right)  $ for which $g_{\mathrm{x}}(q\left(
\mathrm{x}\right)  )=0$. This means that%
\[
F\left(  p^{\ast}+A_{1}\mathrm{x}+A_{2}q\left(  \mathrm{x}\right)  \right)
=g_{\mathrm{x}}(q\left(  \mathrm{x}\right)  )=0.
\]

To prove that $q\left(  \mathrm{x}\right)  $ is smooth, consider $g:X\times
Y\rightarrow\mathbb{R}^{m}$ defined as%
\[
g(\mathrm{x},\mathrm{y})=F\left(  p^{\ast}+A_{1}\mathrm{x}+A_{2}%
\mathrm{y}\right)  .
\]
Since $g(\mathrm{x},\mathrm{y})=g_{\mathrm{x}}(\mathrm{y})$ we see that
$g(\mathrm{x},q\left(  \mathrm{x}\right)  )=0.$ This means that in order to
prove that $q\left(  \mathrm{x}\right)  $ is smooth it is enough to show that
for any $\mathrm{x}\in X$ and $\mathrm{y}\in Y$ the matrix $\frac{\partial
g}{\partial\mathrm{y}}\left(  \mathrm{x},\mathrm{y}\right)  $ is invertible
(smoothness then follows from the implicit function theorem). Since
$\frac{\partial g}{\partial\mathrm{y}}=Dg_{\mathrm{x}}\in\left[  DF_{X,Y}%
A_{2}\right]  $, we see that the matrix must be invertible, since for
$N\left(  \mathrm{x}_{0},\mathrm{y}_{0},A_{1},A_{2},X,Y\right)  $ to be well
defined we have implicitly assumed that any matrix in $\left[  DF_{X,Y}%
A_{2}\right]  $ is invertible.
\end{proof}

Now we comment on the choice of $A_{1}$, $A_{2}$ and discuss why Theorem
\ref{th:directed-Newton} is better than Theorem
\ref{thm:Newton-enclosure-simple}.

\begin{remark}
\label{rem:A1_choice}When $\Sigma_{0}$ is a $C^{1}$ manifold of dimension
$m-k$, then since $F(\Sigma_{0})=0$, we see that for a point $p\in\Sigma_{0}$
the tangent space $T_{p}\Sigma_{0}$ to $\Sigma_{0}$ at $p$ is an $m-k$
dimensional space, which lies in the kernel of $DF\left(  p\right)  $.
We can take $A_{1}$ whose columns consist of vectors which span $T_{p}\Sigma_{0}$.  The image of $A_{1}$ is then $T_{p}%
\Sigma_{0}$, and for any $v\in\mathbb{R}^{m-k}$, $DF\left(  p\right)
A_{1}v=0.$ Then, provided that $X$ is a small set, $\left[  DF_{X,\mathrm{y}%
_{0}}A_{1}\right]  \left[  X-\mathrm{x}_{0}\right]  $ should be small. This
means that by incorporating $A_{1}$ in the setup of local coordinates improves
the deficiency of Theorem \ref{thm:Newton-enclosure-simple}.
\end{remark}

\begin{remark}
Once $A_{1}$ is chosen we can choose $A_{2}$ as any matrix of rank $k$ so that
the image of $A_{2}$ is orthogonal to $A_{1}$.
\end{remark}

\begin{remark}
In practice, the candidate for a set $Y$ can be found automatically by
iterating the operator $N$ several times.
\end{remark}

\begin{remark}
In our computer assisted proof, when applied to the Lomel\'{\i} map, we have
found that Theorem \ref{th:directed-Newton} works better than the direct
approach from Theorem \ref{thm:Newton-enclosure-simple}. To give an indication
of the difference between the two: If we were to take the interval enclosure
of $p^{\ast}+A_{1}\mathrm{x}+A_{2}q(\mathrm{x}),$ i.e.%
\[
\tilde{X}\times\tilde{Y}:=\left[  p^{\ast}+A_{1}X+A_{2}Y\right]  ,
\]
consider a mid point $\mathrm{y}_{0}$ of $\tilde{Y},$ and compute
$N(\mathrm{y}_{0},\tilde{X},\tilde{Y})$, then in our computer assisted
validation the diameter of the set $N(\mathrm{y}_{0},\tilde{X},\tilde{Y})$
turns out to be up to thirty times larger than the diameter of $\tilde{Y}$.
Thus, we would not be able to validate (\ref{eq:Newton-direct}). The validation using Theorem \ref{th:directed-Newton} does go through.

\end{remark}

\begin{corollary}
By differentiating (\ref{eq:zero-with-q}) we see that%
\[
DF\left(  p^{\ast}+A_{1}\mathrm{x}+A_{2}q(\mathrm{x})\right)  \left(
A_{1}+A_{2}Dq\left(  \mathrm{x}\right)  \right)  =0,
\]
hence%
\[
Dq\left(  \mathrm{x}\right)  \in-\left[  DF_{X,Y}A_{2}\right]  ^{-1}\left[
DF_{X,Y}\right]  A_{1}.
\]
This means that we can obtain bounds with computer assistance on the
derivative of $q\left(  \mathrm{x}\right)  $.
\end{corollary}

We shall now use Theorem \ref{th:directed-Newton} to establish intersections
of stable/unstable manifolds. Recall that $F$ was defined using
(\ref{eq:F-def}). Let us introduce the following notation. For $q(\mathrm{x})$
from Theorem \ref{th:directed-Newton} let $\theta:X\rightarrow\mathbb{R}%
^{u_{1}}$ and $\phi:X\rightarrow\mathbb{R}^{s_{2}}$ be defined as%
\begin{align*}
\theta\left(  \mathrm{x}\right)    & =\pi_{\theta}\left(  p^{\ast}%
+A_{1}\mathrm{x}+A_{2}q(\mathrm{x})\right)  ,\\
\phi\left(  \mathrm{x}\right)    & =\pi_{\phi}\left(  p^{\ast}+A_{1}%
\mathrm{x}+A_{2}q(\mathrm{x})\right)  .
\end{align*}
(This means that $(\theta\left(  \mathrm{x}\right),\phi\left(  \mathrm{x}\right))=p^{\ast}+A_{1}%
\mathrm{x}+A_{2}q(\mathrm{x})$.) Also, let $p:X\rightarrow\mathbb{R}^{k}$ be defined as
\[
p\left(  \mathrm{x}\right)  =f^{l_{1}}(P_{1}(\theta\left(  \mathrm{x}\right)
)).
\]

\begin{theorem}
Assume that $F$ is defined by (\ref{eq:F-def}), and that assumptions of
Theorem \ref{th:directed-Newton} are fulfilled. Then the manifolds
$W^{u}\left(  p_{1}\right)  $ and $W^{s}(p_{2})$ intersect transversally
along $p\left(  \mathrm{x}\right)  $ for $\mathrm{x}\in X$. Moreover,
$p\left(  \mathrm{x}\right)  $ is as smooth as $f$.
\end{theorem}

\begin{proof}
By Theorem \ref{th:directed-Newton},
\[
0=F\left(  p^{\ast}+A_{1}\mathrm{x}+A_{2}q(\mathrm{x})\right)  =f^{l_{1}%
}(P_{1}(\theta\left(  \mathrm{x}\right)  ))-f^{-l_{2}}\left(  P_{2}%
(\phi\left(  \mathrm{x}\right)  )\right)  ,
\]
meaning that%
\[
p\left(  \mathrm{x}\right)  :=f^{l_{1}}(P_{1}(\theta\left(  \mathrm{x}\right)
))=f^{-l_{2}}\left(  P_{2}(\phi\left(  \mathrm{x}\right)  )\right)  .
\]
Since $P_{1}(\theta\left(  \mathrm{x}\right)  )\in W^{u}(p_{1})$ and
$P_{2}(\phi\left(  \mathrm{x}\right)  )\in W^{s}(p_{2}),$ we see that
$W^{u}\left(  p_{1}\right)  $ and $W^{s}(p_{2})$ intersect at $p\left(
\mathrm{x}\right)  $. 

We now address the issue of transversality. Consider a $k\times u_{1}$ matrix
$C_{1}$ and a $k\times s_{2}$ matrix $C_{2}$ defined as%
\begin{align*}
C_{1} &  =\frac{d}{d\theta}f^{l_{1}}(P_{1}(\theta\left(  \mathrm{x}\right)
)),\\
C_{2} &  =\frac{d}{d\phi}f^{-l_{2}}(P_{2}(\phi\left(  \mathrm{x}\right)  )),
\end{align*}
and observe that%
\begin{align*}
T_{p\left(  \mathrm{x}\right)  }W^{u}\left(  p_{1}\right)   &  =\left\{
C_{1}v:v\in\mathbb{R}^{u_{1}}\right\}  ,\\
T_{p\left(  \mathrm{x}\right)  }W^{s}\left(  p_{2}\right)   &  =\left\{
C_{2}w:w\in\mathbb{R}^{s_{2}}\right\}  .
\end{align*}
To prove that $W^{u}\left(  p_{1}\right)  $ and $W^{s}\left(  p_{2}\right)  $
intersect transversally, we need to show that
\begin{equation}
\left\{  C_{1}v+C_{2}w:v\in\mathbb{R}^{u_{1}},w\in\mathbb{R}^{s_{2}}\right\}
=\mathbb{R}^{k}.\label{eq:trans-span}%
\end{equation}
Since assumptions of Theorem \ref{th:directed-Newton} hold, the matrix
$DF\left(  \theta\left(  \mathrm{x}\right)  ,\phi\left(  \mathrm{x}\right)
\right)  A_{2}$ is invertible. Observing that
\[
DF\left(  \theta\left(  \mathrm{x}\right)  ,\phi\left(  \mathrm{x}\right)
\right)  A_{2}=C_{1}\pi_{\theta}A_{2}-C_{2}\pi_{\phi}A_{2},
\]
invertibility implies that for any $p\in\mathbb{R}^{k}$ there exists a
$\mathrm{y}\in\mathbb{R}^{k}$ such that
\[
C_{1}\pi_{\theta}A_{2}\mathrm{y}-C_{2}\pi_{\phi}A_{2}\mathrm{y}=p.
\]
Above equation implies (\ref{eq:trans-span}). 

The $p\left(  \mathrm{x}\right)  $ is smooth, since it is a composition of
smooth functions. This concludes the proof.
\end{proof}

\begin{figure}[ptb]
\begin{center}
\includegraphics[height=3cm]{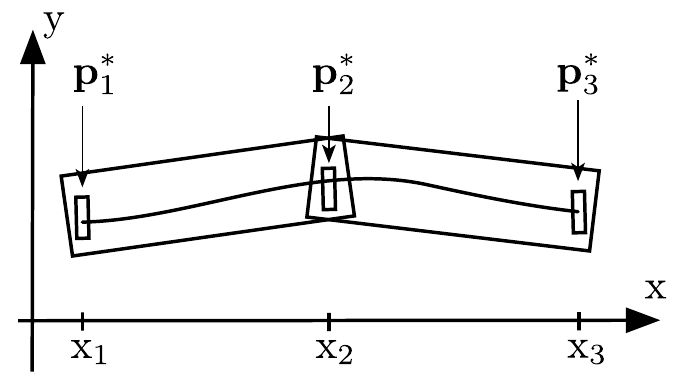}
\end{center}
\caption{The enclosure of a curve from section \ref{sec:curve-bound}. }%
\label{fig:curve-bound}%
\end{figure}

\subsection{Establishing intersections of manifolds along
curves\label{sec:curve-bound}}

We now show how to use the method to establish a bound for a curve along which
$F$ is zero. This example will later be used by us to establish one
dimensional curves along intersections of $W^{u}(p_{1})$ and $W^{s}(p_{2})$ in
the Lomel\'{\i} map.

We will use bold font to denote interval enclosures of sets. This means that
all notations in bold represent interval sets (cubes), and operations
performed on them are in interval arithmetic.

Let $B_{1},B_{2},...,B_{N+1}$ be a sequence of cubes in $\mathbb{R}^{m-1}$,
and let $\mathrm{x}_{1},\mathrm{x}_{2},\ldots,\mathrm{x}_{N+1}\in\mathbb{R}$.
We consider a sequence of sets $\mathbf{p}_{1}^{\ast},\ldots,\mathbf{p}%
_{N+1}^{\ast},$ of the form
\begin{equation}
\mathbf{p}_{n}^{\ast}=\left\{  \mathrm{x}_{n}\right\}  \times B_{n}%
\qquad\text{for }n=1,\ldots,N+1. \label{eq:p-star-form}%
\end{equation}
(See Figure \ref{fig:curve-bound}.) Using Theorem
\ref{thm:Newton-enclosure-simple} (taking $X=\{\mathrm{x}_{n}\}$ and $Y=B_{n}%
$), we can establish that $\mathbf{p}_{n}^{\ast}$ contain zeros of $F$. Our
objective will be to obtain a bound on the curve along which $F$ is zero,
which joins the points in $\mathbf{p}_{1}^{\ast},\ldots,\mathbf{p}_{N+1}%
^{\ast}$.

Let $\mathbf{A}_{1,n}=\mathbf{p}_{n+1}^{\ast}-\mathbf{p}_{n}^{\ast}$. Consider
$X=\left[  0,1\right]  $, consider a sequence of closed $m-1$ dimensional
cubes $Y_{1},...,Y_{N}$ in $\mathbb{R}^{m-1}$ and a sequence of matrixes
$A_{2,1},\ldots,A_{2,N}.$ We can choose these so that the range of $A_{2,n}$
is (roughly) orthogonal to the range of $\mathbf{A}_{1,n}$, for $n=1,...,N$.
Such choice can easily be automated. The choice of $Y_{n}$ can also be
automated, by iterating the Newton operator defined in
(\ref{eq:Newton-operator}).

\begin{remark}
Note that $\mathbf{p}_{n+1}^{\ast},\mathbf{p}_{n}^{\ast}$ are on the curve
which we wish to establish. This means that $\mathbf{A}_{1,n}=\mathbf{p}%
_{n+1}^{\ast}-\mathbf{p}_{n}^{\ast}$ is close to the tangent space of the
curve. This by Remark \ref{rem:A1_choice} means that such $\mathbf{A}_{1,n}$
should be a good choice, meaning that we should have
\[
[DF\left(  \mathbf{p}_{n}^{\ast}\right)] \mathbf{A}_{1,n} \mathrm{x} \approx0,
\qquad\mbox{for } \mathrm{x}\in\left[  0,1\right] .
\]

\end{remark}

Let $\mathrm{x}_{0}\in X,$ $\mathrm{y}_{0,n}\in Y_{n}$ be the mid points of
the sets $X$ and $Y_{n}$, respectively. Assume that for any $A_{1,n}%
\in\mathbf{A}_{1,n}$ and $p_{n}^{\ast}\in\mathbf{p}_{n}^{\ast}$, assumptions
of Theorem \ref{th:directed-Newton} hold for $X,Y_{n},x_{0},\mathrm{y}_{0}%
^{n},A_{1,n},A_{2,n},p_{n}^{\ast}$. If this is true for
$n=1,\ldots,N$, then from Theorem \ref{th:directed-Newton} it follows that
there exists a curve joining the points in $\mathbf{p}_{1}^{\ast},\ldots
,\mathbf{p}_{N+1}^{\ast}$ on which $F$ is zero. The curve is contained in the
set
\[
\bigcup_{n=1}^{N}\left\{  \mathbf{p}_{n}^{\ast}+\mathbf{A}_{1,n}%
\mathrm{x}+A_{2,n}\mathrm{y}:\mathrm{x}\in\left[  0,1\right]  ,\mathrm{y}\in
Y_{n}\right\}  .
\]
Above procedure can be summed up as follows:\smallskip

\noindent \hrulefill

\noindent \textbf{Algorithm.}\label{algorithm}\smallskip

\noindent\textbf{Input: }A sequence of points $p_{1},\ldots,p_{N+1}$, for
which $F\left(  p_{n}\right)  \approx0$, for $n=1,\ldots,N+1$. (These points
can be computed non-rigorously.)\smallskip

\noindent\textbf{Output:} A sequence of sets:%
\[
\left\{  \mathbf{p}_{n}^{\ast}+\mathbf{A}_{1,n}\mathrm{x}+A_{2,n}%
\mathrm{y}:\mathrm{x}\in\left[  0,1\right]  ,\mathrm{y}\in Y_{n}\right\}  ,
\]
which enclose the curve on which $F$ is zero.\smallskip

\noindent\textbf{Steps:}

\begin{enumerate}
\item Enclose $p_{1},\ldots,p_{N+1}$ in sets $\mathbf{p}_{1}^{\ast}%
,\ldots,\mathbf{p}_{N+1}^{\ast}$ of the form (\ref{eq:p-star-form}) and
validate existence of zeros of $F$ inside of $\mathbf{p}_{1}^{\ast}%
,\ldots,\mathbf{p}_{N+1}^{\ast}$ using Theorem
\ref{thm:Newton-enclosure-simple} and Remark \ref{rem:zeros-pts}.

\item For $\mathbf{A}_{1,n}=\mathbf{p}_{n+1}^{\ast}-\mathbf{p}_{n}^{\ast}$,
choose $m\times m-1$ matrixes, for which the range of $A_{2,i}$ is (roughly)
orthogonal to $\mathbf{A}_{1,n}$.

\item Take $X=\left[  0,1\right]  ,$ $\mathrm{x}_{0}=\frac{1}{2},$
$\mathrm{y}_{0}=0$ and $Y_{1}=...=Y_{N}=\left\{  0\right\}  .$ Iterate the
Newton operator (\ref{eq:Newton-operator}) several times and enlarge each
$Y_{n}$.

\item For $n=1,...,N$, validate assumptions of Theorem
\ref{th:directed-Newton} for
\[
X,Y_{n},x_{0},\mathrm{y}_{0}^{n},\mathbf{A}_{1,n},A_{2,n},\mathbf{p}_{n}%
^{\ast}.
\]

\end{enumerate}

\noindent \hrulefill

\subsection{Application to the Lomel\'{\i} map\label{sec:Lom-application}}

In this section we apply the method from section \ref{sec:curve-bound} to the
Lomel\'{\i} map. Here we establish a computer assisted proof of two types of
intersections. The first type is when the stable and unstable manifolds
intersect along closed curves, which is the setting from Figure
\ref{fig:LomeliMap2}. Such intersections are established in section
\ref{sec:loop-CAP}. The second type of intersection is along heteroclinic
arcs, as is the case in Figure \ref{fig:LomeliMap1}. Such arcs are established
in section \ref{sec:arc-CAP}.

\begin{figure}[ptb]
\begin{center}
\includegraphics[height=4.5cm]{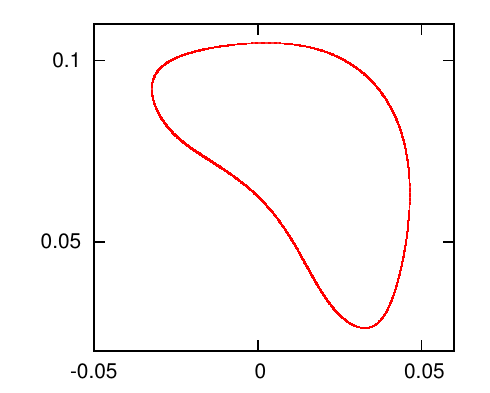}
\includegraphics[height=4.5cm]{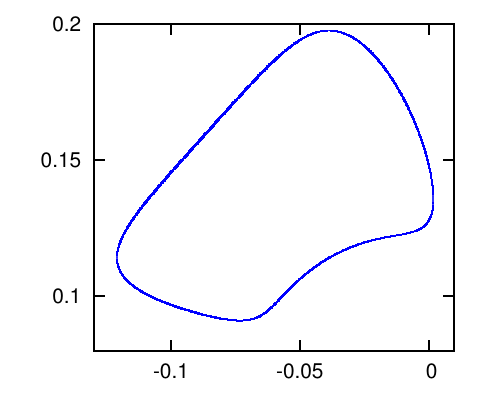}
\end{center}
\caption{The enclosure of a heteroclinic loop in the parameter space. On the
left we have the projection onto $B_{s}$ (in red) and on the right the
projection onto $B_{u}$ (in blue). Here we considered the Lomel\'{\i} map with
$a=\frac{1}{2},$ $b=-\frac{1}{2},$ $c=1,$ $\alpha=-0.08999$, $\tau=\frac
{8}{10}$.}%
\label{fig:loop}%
\end{figure}

\subsubsection{Heteroclinic loops}

\label{sec:loop-CAP}

In this section we consider the Lomel\'{\i} map (\ref{eq:LomeliMap}) with
parameters $a=\frac{1}{2},$ $b=-\frac{1}{2},$ $c=1,$ $\alpha=-0.08999$,
$\tau=\frac{8}{10}$ and give a computer assisted proof of a connection of
$W^{u}\left(  p_{1}\right)  $ with $W^{s}(p_{2})$ along a heteroclinic loop.
Such loop generates, by iterates of $f$, the intersections of $W^{u}(p_{1})$
and $W^{s}(p_{2})$, which are shown in Figure \ref{fig:LomeliMap2}.

We consider
\[
F:B_{u}\times B_{s}\rightarrow\mathbb{R}^{3}%
\]
defined as%
\begin{equation}
F(\theta,\phi)=f^{l_{1}}(P_{1}(\theta))-f^{-l_{2}}\left(  P_{2}(\phi)\right)
, \label{eq:F-Lom-appl}%
\end{equation}
with $l_{1}=l_{2}=9.$

\begin{remark}
We point out that (\ref{eq:F-Lom-appl}) involves many compositions of the map
$f$. Computing such compositions directly in interval arithmetic leads to a
blowup. This is associated with the fact that enclosing each iterate in a rectangular box produced overestimates; the so called \emph{wrapping effect}. (For more information on the wrapping effect see \cite{MR551212}.)
 If we were to naively compose $f$
in interval arithmetic, then the below obtained results would not go through.
For our computation of interval enclosure of $F$ and $DF$ we use a careful,
Lohner-type set representation, that reduces the wrapping effect when
computing bounds on $f^{k}$ and $Df^{k}$. This representation is discussed in
detail in section \ref{sec:wrapping}.
\end{remark}

We follow the procedure from section \ref{sec:curve-bound} to establish the
enclosure of the curve in the parameter space $B_{u}\times B_{s}$. Our curve
is enclosed using $N=1309$ small cubes in $\mathbb{R}^{4}$. Figure
\ref{fig:loop} consists of these cubes, but this is not visible from the plot.
After magnification the cubes start to take shape. In Figure
\ref{fig:loop_closeup} we show a close-up of $50$ of such cubes. (The total number  of considered cubes results from a non-rigorous procedure, which we have used to find initial points that are close to the intersection of the manifolds. They are the input of the Algorithm from page \pageref{algorithm}. The $N=1309$ is arbitrary, and we could have chosen a different number.)

We can propagate the loop in the parameter space using the linear inner dynamics.
This way we obtain the plot from Figure \ref{fig:loop_all}. This figure
corresponds to the intersections visible in Figure \ref{fig:LomeliMap2}. The
difference is that the loops in Figure \ref{fig:loop} are in the parameter
space and the loops from Figure \ref{fig:LomeliMap2} are in the state space. Our bounds on the inner dynamics are rigorous. They follow from the parameterisation method (see (\ref{eq:parmInvEq}) and Remark \ref{rem:unstableParm}). Thus, the resulting plots from Figure \ref{fig:loop_all} are also rigorous enclosures of the curves, as long as we remain within the domains of our parameterisations. The boundary of the domain is depicted in green. Thus, the part of the plot which is within the green boundary is rigorous.

The computation needed to establish the hetoerclinic loop took 59 seconds on a single 3GHz Intel i7 core processor, running on MacBook Pro, with OS X 10.9.5. We have conducted our proof using c++ and the CAPD\footnote{Computer Assisted Proofs in Dynamics: http://capd.ii.uj.edu.pl/} library.  For the 
computational environment used to obtain 
the local stable/unstable parameterizations 
we refer to Remark \ref{rem:parmComps}.

\begin{figure}[ptb]
\begin{center}
\includegraphics[height=4.5cm]{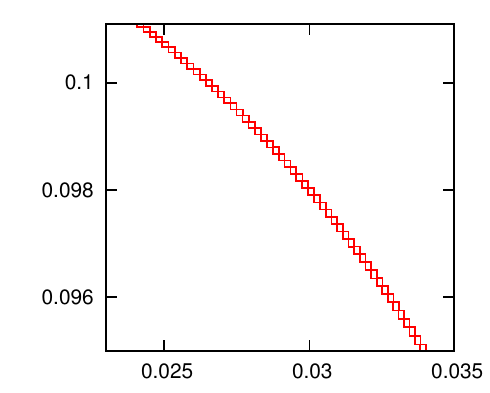}
\includegraphics[height=4.5cm]{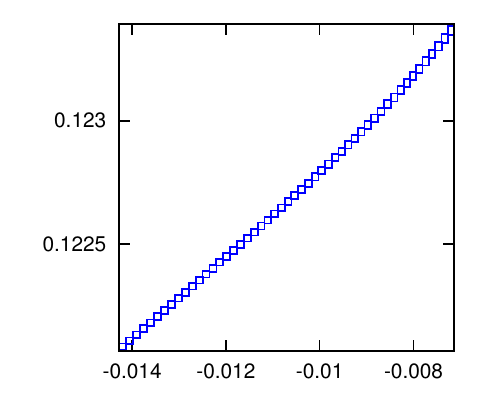}
\end{center}
\caption{The closeup of the enclosure of a heteroclinic loop from Figure
\ref{fig:loop}.}%
\label{fig:loop_closeup}%
\end{figure}

\begin{figure}[ptb]
\begin{center}
\includegraphics[height=4.5cm]{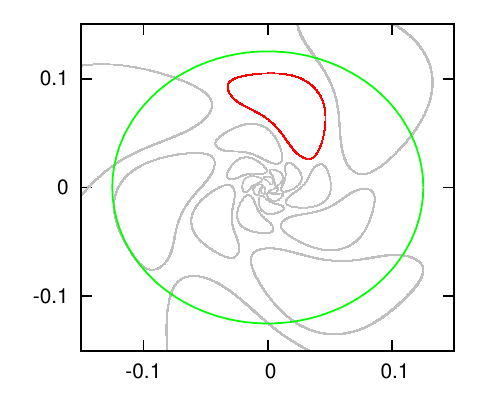}
\includegraphics[height=4.5cm]{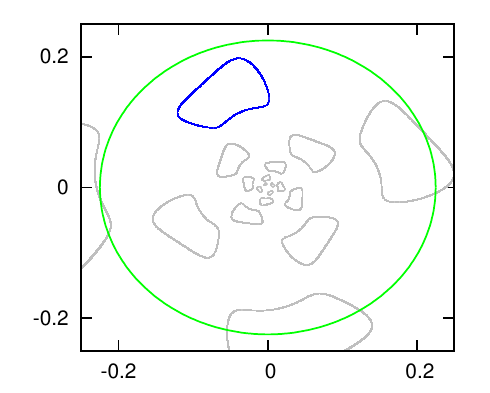}
\end{center}
\caption{The loop from Figure \ref{fig:loop} propagated using the linear inner
dynamics on the parameter space. In green we have the boundaries of the
domains of the parameterisations.}%
\label{fig:loop_all}%
\end{figure}

\subsubsection{Heteroclinic arcs}

\label{sec:arc-CAP}

In this section we consider the Lomel\'{\i} map with parameters $a=\frac
{44}{100},$ $b=\frac{21}{100},$ $c=\frac{35}{100},$ $\alpha=-\frac{1}{4}$,
$\tau=-\frac{3}{10}$ and give a computer assisted proof of a connection of
$W^{u}\left(  p_{1}\right)  $ with $W^{s}(p_{2})$ along two $3$-fold primary
heteroclinic arcs, which lead to six homoclinic paths from $p_{1}$ to $p_{2}$.

We consider $F$ as defined in (\ref{eq:F-Lom-appl}). Applying Theorem
\ref{thm:Newton-enclosure-simple} together with Remark \ref{rem:zeros-pts}, we
establish $N=309$ enclosures of points $\mathbf{p}_{1}^{\ast},\ldots
,\mathbf{p}_{N}^{\ast}$ on which $F$ is zero. (The number of considered points is arbitrary; as long as the validation would go through we could take a different $N$.) The dynamics on the unstable and
stable manifolds is conjugated with a linear map
\begin{align*}
f\circ P_{1}\left(  \theta\right)   &  =P_{1}\left(  A_{1}\theta\right)  ,\\
f\circ P_{2}(\phi)  &  =P_{2}\left(  A_{2}\phi\right)  ,
\end{align*}
where%
\[
A_{i}=\left(
\begin{array}
[c]{ll}%
\mathrm{re}\lambda_{i} & -\mathrm{im}\lambda_{i}\\
\mathrm{im}\lambda_{i} & \mathrm{re}\lambda_{i}%
\end{array}
\right)  \qquad\text{for }i=1,2,
\]
and $\lambda_{1},\lambda_{2}$ are the eigenvalues of $Df\left(  p_{1}\right)
$ and $Df\left(  p_{2}\right)  $, respectively. We have established the
following bounds for the eigenvalues
\begin{align*}
\mathrm{re}\lambda_{1}  &  \in\mathbf{re}\boldsymbol{\lambda}_{1}=\left[
-0.71570025199987,-0.71570025199985\right]  ,\\
\mathrm{im}\lambda_{1}  &  \in\mathbf{im}\boldsymbol{\lambda}_{1}=\left[
-0.93025058966104,-0.93025058966103\right]  ,
\end{align*}%
\begin{align*}
\mathrm{re}\lambda_{2}  &  \in\mathbf{re}\boldsymbol{\lambda}_{2}=\left[
-0.47875667823481,-0.47875667823480\right]  ,\\
\mathrm{im}\lambda_{2}  &  \in\mathbf{im}\boldsymbol{\lambda}_{2}=\left[
-0.70015090953401,-0.70015090953400\right]  .
\end{align*}
We consider interval matrixes
\begin{align*}
\mathbf{A}_{i}  &  =\left(
\begin{array}
[c]{ll}%
\mathbf{re}\boldsymbol{\lambda}_{i} & -\mathbf{im}\boldsymbol{\lambda}_{i}\\
\mathbf{im}\boldsymbol{\lambda}_{i} & \mathbf{re}\boldsymbol{\lambda}_{i}%
\end{array}
\right)  ,\qquad\text{for }i=1,2,\\
\mathbf{B}  &  \mathbf{=}\left(
\begin{array}
[c]{ll}%
\mathbf{A}_{1} & 0\\
0 & \mathbf{A}_{2}%
\end{array}
\right)  ,
\end{align*}
and take
\[
\mathbf{p}_{N+1}^{\ast}=\mathbf{B}^{3}\mathbf{p}_{1}^{\ast}.
\]

Following the method from section \ref{sec:curve-bound}, we establish an
enclosure of a curve (in parameter space) $\gamma=\left(  \gamma_{u}%
,\gamma_{s}\right)  \subset B_{u}\times B_{s}$, which passes through
$\mathbf{p}_{1}^{\ast},\mathbf{p}_{2}^{\ast},\ldots,\mathbf{p}_{N+1}^{\ast}.$
The enclosure is shown in Figure \ref{fig:Arc1}.

\begin{figure}[ptb]
\begin{center}
\includegraphics[height=4.5cm]{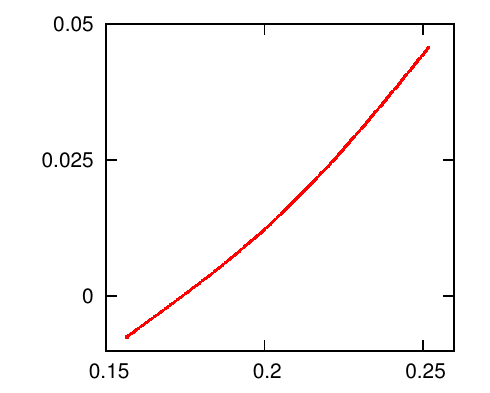}
\includegraphics[height=4.5cm]{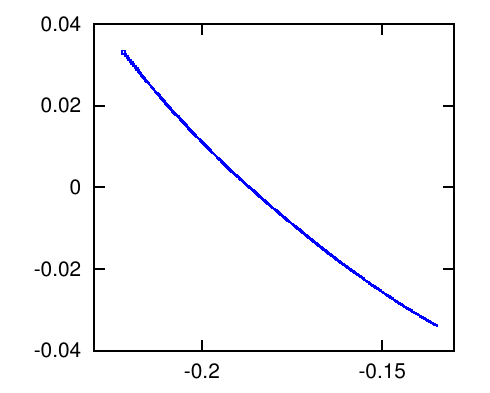}
\end{center}
\caption{Enclosure of a fundamental heteroclinic arc. On the left we have the
projection onto $B_{s}$ (in red) and on the right the projection onto $B_{u}$
(in blue). Here we considered the Lomel\'{\i} map with $a=\frac{44}{100},$
$b=\frac{21}{100},$ $c=\frac{35}{100},$ $\alpha=-\frac{1}{4}$ and $\tau
=-\frac{3}{10}$. }%
\label{fig:Arc1}%
\end{figure}The path $\gamma$ can be iterated by the linear dynamics in the
parameter space. A couple of such iterates result in a picture from Figure
\ref{Fig:Arc1-param}.

\begin{figure}[ptb]
\begin{center}
\includegraphics[height=4.5cm]{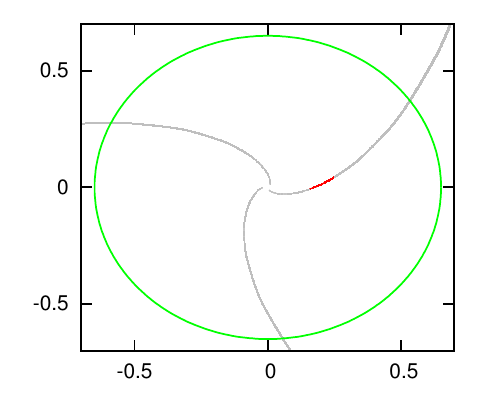}
\includegraphics[height=4.5cm]{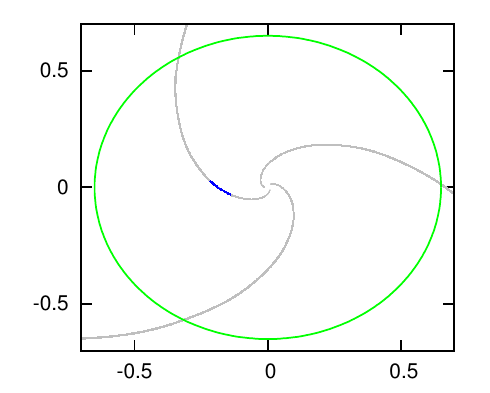}
\end{center}
\caption{Enclosures of heteroclinic paths in parameter space. In blue and red
we have the fundamental heteroclinic arc from Figure \ref{fig:Arc1}. All three
gray homoclinic paths are obtained by iterates of the single fundamental
heteroclinic arc. In green we have the boundaries of the domains of the
parameterisations.}%
\label{Fig:Arc1-param}%
\end{figure}

Since we took $\mathbf{p}_{N+1}^{\ast}=\mathbf{B}^{3}\mathbf{p}_{1}^{\ast},$
the $P_{1}\left(  \gamma_{u}\right)  $ and $P_{2}\left(  \gamma_{s}\right)  $
are $3$-fold fundamental heteroclinic arcs (as discussed in case 2 from
Section \ref{sec:heteroArcs}) and we obtain a heteroclinic path:%
\[
S_{3}=S_{3}\left(  P_{1}\left(  \gamma_{u}\right)  \right)  =\bigcup
_{i\in\mathbb{Z}}f^{3i}\left(  P_{1}\left(  \gamma_{u}\right)  \right)  .
\]

\begin{remark}
Since we know that $F\left(  \gamma\right)  =0,$ by (\ref{eq:F-Lom-appl}),
\[
f^{9}(P_{1}(\gamma_{u}))-f^{-9}\left(  P_{2}(\gamma_{s})\right)  =0,
\]
and thus
\[
S_{3}=\bigcup_{i\in\mathbb{Z}}f^{3i}\left(  P_{1}\left(  \gamma_{u}\right)
\right)  =\bigcup_{i\in\mathbb{Z}}f^{3i}\left(  P_{2}\left(  \gamma
_{s}\right)  \right)  .
\]

\end{remark}

Also $f(S_{3})$ and $f^{2}(S_{3})$ are heteroclinic paths. Thus the $3$-fold
fundamental heteroclinic arc $\gamma$ generates three paths. These three paths
lie on the intersection of the stable and unstable manifolds from Figure
\ref{fig:LomeliMap1}.

For the investigated parameters, one can find a second $3$-fold fundamental
heteroclinic arc that generates a different set of three homoclinic paths. We
have obtained its enclosure using the same procedure, by considering $N=344$
cubes. The obtained enclosure is given in Figure \ref{fig:second-arc}. We can
propagate this arc using the linear inner dynamics on the parameter space (see (\ref{eq:parmInvEq}) and Remark \ref{rem:unstableParm}), and
obtain the plot in Figure \ref{fig:all-arcs}. Since we have rigorous bounds on the inner dynamics, the resulting plots are rigorous, as long as the estimates stay within the domains of the parameterisations.

As a result, we obtain six homoclinic paths in total, which form the
intersections of the stable and unstable manifolds shown in Figure
\ref{fig:LomeliMap1}. The paths from Figure \ref{fig:all-arcs} are in the
parameter space, and the paths in Figure \ref{fig:LomeliMap1} are in the state space.

\begin{figure}[ptb]
\begin{center}
\includegraphics[height=4.5cm]{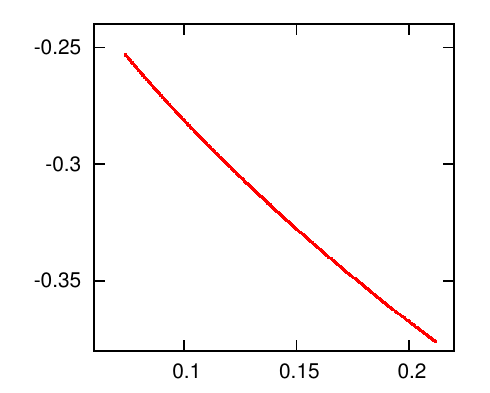}
\includegraphics[height=4.5cm]{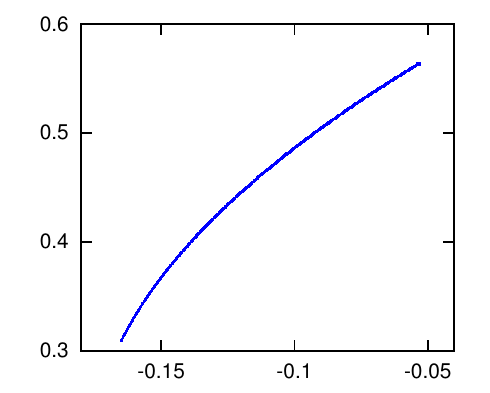}
\end{center}
\caption{Enclosure of the second fundamental heteroclinic arc for the
Lomel\'{\i} map with $a=\frac{44}{100},$ $b=\frac{21}{100},$ $c=\frac{35}%
{100},$ $\alpha=-\frac{1}{4}$ and $\tau=-\frac{3}{10}$. On the left we have
the projection onto $B_{s}$ (in red) and on the right the projection onto
$B_{u}$ (in blue).}%
\label{fig:second-arc}%
\end{figure}

\begin{figure}[ptb]
\begin{center}
\includegraphics[height=4.5cm]{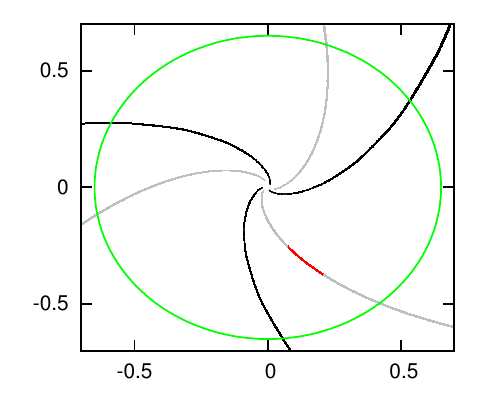}
\includegraphics[height=4.5cm]{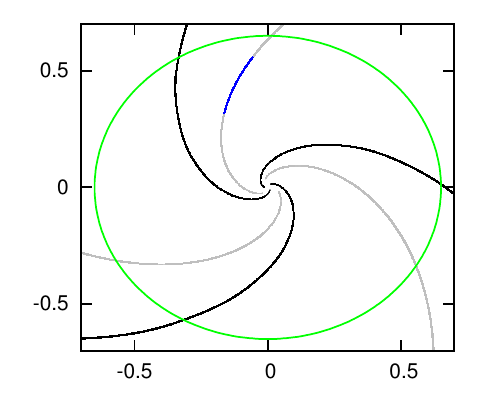}
\end{center}
\caption{In gray, the enclosures of three homoclinic paths generated by the
second fundamental heteroclinic arc. In black, the previous three paths from
Figure \ref{Fig:Arc1-param}. The second fundamental heteroclinic arc from
Figure \ref{fig:second-arc} is in red and blue. In green we have the
boundaries of the domains of the parameterisations.}%
\label{fig:all-arcs}%
\end{figure}

The computation needed to establish the two fundamental heteroclinic arcs took
27 seconds on a single 3GHz Intel i7 core processor, running on MacBook Pro, with OS X 10.9.5. Our proof has been implemented in c++, using the CAPD\footnote{Computer Assisted Proofs in Dynamics: http://capd.ii.uj.edu.pl/} library.
For the 
computational environment used to obtain 
the local stable/unstable parameterizations 
we refer to Remark \ref{rem:parmComps}.

\subsection{Controlling the wrapping effect\label{sec:wrapping}}

In the computer assisted proofs from Sections \ref{sec:loop-CAP},
\ref{sec:arc-CAP} we have established the connections of the manifolds by
investigating $F=0$, where $F$ was given by (\ref{eq:F-Lom-appl}). This $F$
involves many compositions of the Lomel\'{\i} map $f$. To apply Theorem
\ref{th:directed-Newton} we need good estimates on interval enclosure of $F$
and $DF$ computed on sets. If this is computed directly in interval arithmetic
by composing $f$ and $Df$ many times, then the wrapping effect significantly
reduces the accuracy of the computations. In order to overcome the wrapping
effect, one can use multiple shooting or good set representation. In our
approach we have chosen the second approach.

We use a Lohner type representation for images and derivatives by the map $f$.
The approach is in the spirit of \cite{MR1930946}, but simpler since we
consider a map instead of integrating an ODE.

To discuss our set representation we need some auxiliary notations. In this
section we shall use the calligraphic font $\mathcal{A}$ to denote sequences
of matrixes%
\[
\mathcal{A}=\left(  A_{1},\ldots,A_{n}\right)  .
\]
In our setting, $A_{i}$ will be $n\times n$ Hessians matrixes. (For the
Lomel\'{\i} map $n=3$.) For $b\in\mathbb{R}^{n}$ we will write $b^{T}%
\mathcal{A}$ to denote an $n\times n$ matrix and $b^{T}\mathcal{A}b$ to denote
a vector, defined as follows:
\begin{equation}
b^{T}\mathcal{A}=\left(
\begin{array}
[c]{c}%
b^{T}A_{1}\\
\vdots\\
b^{T}A_{n}%
\end{array}
\right)  ,\qquad\text{\qquad}b^{T}\mathcal{A}b=\left(
\begin{array}
[c]{c}%
b^{T}A_{1}b\\
\vdots\\
b^{T}A_{n}b
\end{array}
\right)  . \label{eq:b-prod-hess}%
\end{equation}
(Each $b^{T}A_{i}$ is a $1\times n$ matrix, which constitutes the $i$-th row
of the $n\times n$ matrix $b^{T}\mathcal{A}$.) We will also use a convention
in which for a matrix $B$ we shall write $B\mathcal{A}$ and $\mathcal{A}B$ for
sequences of matrixes defined as follows:
\[
B\mathcal{A=}\left(  BA_{1},\ldots,BA_{n}\right)  \qquad\text{and\qquad
}\mathcal{A}B=\left(  A_{1}B,\ldots,A_{n}B\right)  .
\]
For an $n\times n$ matrix $B=(B_{ki})_{k,i=1}^{n}$ we define a sequence of
matrixes $B\ast\mathcal{A}$ as follows
\[
B\ast\mathcal{A=}\left(  \sum_{i=1}^{n}B_{1i}A_{i},\ldots,\sum_{i=1}^{n}%
B_{ni}A_{i}\right)  .
\]

We now give a technical lemma, which will be useful later on.

\begin{lemma}
\label{lem:Hess-manip}For any $\mathcal{A}$, $B$ and $b,$%
\begin{equation}
B\left(  b^{T}\mathcal{A}\right)  =b^{T}\left(  B\ast\mathcal{A}\right)  .
\label{eq:Hess-manip}%
\end{equation}

\end{lemma}

\begin{proof}
The result follows from direct computation. We write out the proof in the appendix.
\end{proof}

For $h:\mathbb{R}^{n}\rightarrow\mathbb{R}$, let $\nabla h$ stand for the
gradient of $h$ and let $\nabla^{2}h$ stand for the Hessian of $h$.

The Lomel\'{\i} map $f$ is quadratic, which means that
\begin{align*}
f_{i}(x_{0}+b)  &  =f_{i}\left(  x_{0}\right)  +Df_{i}\left(  x_{0}\right)
b+\frac{1}{2}b^{T}\left(  \nabla^{2}f_{i}\right)  b,\\
\nabla f_{i}(x_{0}+b)  &  =\nabla f_{i}\left(  x_{0}\right)  +\left(
\nabla^{2}f_{i}\right)  b.
\end{align*}
If we choose $\mathcal{H}=\left(  \nabla^{2}f_{1},\ldots,\nabla^{2}%
f_{n}\right)  ,$ then using our notations (\ref{eq:b-prod-hess}), we can
rewrite the above as (for the second equality below we use the fact that
$Df_{i}\left(  x\right)  =\left(  \nabla f_{i}(x)\right)  ^{T}$)
\begin{align}
f(x_{0}+b)  &  =f\left(  x_{0}\right)  +Df\left(  x_{0}\right)  b+\frac{1}%
{2}b^{T}\mathcal{H}b,\nonumber\\
Df\left(  x_{0}+b\right)   &  =Df\left(  x_{0}\right)  +b^{T}\mathcal{H}.
\label{eq:Df-using-Hess}%
\end{align}

In our computer assisted consideration, we will represent points as
\[
x_{0}+Ab+r,
\]
and derivatives as%
\[
X_{0}+b^{T}\mathcal{A}+R,
\]
meaning that in interval representation%
\begin{align}
x  &  \in x_{0}+A\mathbf{b}+\mathbf{r,}\label{eq:Lohner-set}\\
X  &  \in X_{0}+\mathbf{b}^{T}\mathcal{A}+\mathbf{R}, \label{eq:Lohner-der}%
\end{align}
where $\mathbf{b}$ and $\mathbf{r}$ are interval vectors and $\mathbf{R}$ is
an interval matrix. Below we show why such representation is a good idea. (Its
main objective is to reduce the wrapping effect.)

We take a point of the form
\begin{equation}
x=x_{0}+u\qquad\text{with }u=Ab+r, \label{eq:x-point}%
\end{equation}
and compute%
\begin{align*}
f_{i}\left(  x\right)  =  &  f_{i}\left(  x_{0}\right)  +Df_{i}\left(
x_{0}\right)  u+\frac{1}{2}u^{T}\left(  \nabla^{2}f_{i}\right)  u\\
=  &  f_{i}\left(  x_{0}\right)  +Df_{i}\left(  x_{0}\right)  Ab
 +Df_{i}\left(  x_{0}\right)  r+\frac{1}{2}\left(  Ab+r\right)  ^{T}\left(
\nabla^{2}f_{i}\right)  \left(  Ab+r\right)  .
\end{align*}
From the above we see that if $b\in\mathbf{b}$ and $r\in\mathbf{r}$, then%
\begin{align*}
f\left(  x\right)  \in &  f\left(  x_{0}\right)  +\left(  Df\left(
x_{0}\right)  A\right)  \mathbf{b}  +Df\left(  x_{0}\right)  \mathbf{r}+\frac{1}{2}\left(  A\mathbf{b}%
+\mathbf{r}\right)  ^{T}\mathcal{H}\left(  A\mathbf{b}+\mathbf{r}\right)  ,
\end{align*}
meaning that
\begin{equation}
f\left(  x\right)  \in\tilde{x}_{0}+\tilde{A}\mathbf{b+\tilde{r}},
\label{eq:f-enclosure-rep}%
\end{equation}
for
\begin{align*}
\tilde{x}_{0}  &  =f\left(  x_{0}\right)  ,\\
\tilde{A}  &  =Df\left(  x_{0}\right)  A,\\
\mathbf{\tilde{r}}  &  =Df\left(  x_{0}\right)  \mathbf{r}+\frac{1}{2}\left(
A\mathbf{b}+\mathbf{r}\right)  ^{T}\mathcal{H}\left(  A\mathbf{b}%
+\mathbf{r}\right)  .
\end{align*}

\begin{remark}
\label{rem:Num-reminder}We note that the second term in $\mathbf{\tilde{r}}$
is $O(\left\Vert A\mathbf{b}+\mathbf{r}\right\Vert ^{2}).$ This means that if
$\mathbf{r}$ is $O(\left\Vert \mathbf{b}\right\Vert ^{2}),$ then
$\mathbf{\tilde{r}}$ will also be $O(\left\Vert \mathbf{b}\right\Vert ^{2})$.
The $\tilde{x}_{0}$ can be computed with good accuracy, since we do not need
to evaluate $f$ on a large set, but just at a single point. Similarly,
$Df\left(  x_{0}\right)  $ can be computed accurately, and so in turn will be
$\tilde{A}$.
\end{remark}

\begin{remark}
If we split up our computation of $\mathbf{\tilde{r}}$ into%
\[
\mathbf{\tilde{r}}=Df\left(  x_{0}\right)  \mathbf{r}+\frac{1}{2}\left(
\mathbf{b}^{T}\left(  A^{T}\mathcal{H}A\right)  \mathbf{b}+\mathbf{b}%
^{T}\left(  A^{T}\mathcal{H}\right)  \mathbf{r}+\mathbf{r}^{T}\left(
\mathcal{H}A\right)  \mathbf{b}+\mathbf{r}^{T}\mathcal{H}\mathbf{r}\right)  ,
\]
then the wrapping effect will be reduced even more.
\end{remark}

We now show how our representation works when computing derivatives. We
consider $x$ of the form (\ref{eq:x-point}) and a matrix $X$ of the form
\[
X=X_{0}+b^{T}\mathcal{A}+R.
\]
Below we will compute $Df(x)X.$ Using (\ref{eq:Df-using-Hess}) for the first
equality below, and Lemma \ref{lem:Hess-manip} for the third, we see that%
\begin{align*}
Df(x)X  &  =\left(  Df\left(  x_{0}\right)  +\left(  \left(  Ab+r\right)
^{T}\mathcal{H}\right)  \right)  X\\
&  =Df\left(  x_{0}\right)  \left[  X_{0}+b^{T}\mathcal{A}+R\right]  +\left(
\left(  Ab+r\right)  ^{T}\mathcal{H}\right)  \left[  X_{0}+b^{T}%
\mathcal{A}+R\right] \\
&  =Df\left(  x_{0}\right)  X_{0}+b^{T}\left(  Df\left(  x_{0}\right)
\ast\mathcal{A}+A^{T}\mathcal{H}X_{0}\right) \\
&  \quad+r^{T}\mathcal{H}X_{0}+Df\left(  x_{0}\right)  R+\left(  \left(
Ab+r\right)  ^{T}\mathcal{H}\right)  \left(  b^{T}\mathcal{A}+R\right)  .
\end{align*}
This means that for $b\in\mathbf{b}$, $R\in\mathbf{R}$ and $r\in\mathbf{r}$%
\begin{equation}
Df(x)X\in\tilde{X}_{0}+\mathbf{b}^{T}\widetilde{\mathcal{A}}+\mathbf{\tilde
{R},} \label{eq:Df-enclosure-rep}%
\end{equation}
for%
\begin{align*}
\tilde{X}_{0}  &  =Df\left(  x_{0}\right)  X_{0},\\
\widetilde{\mathcal{A}}  &  =Df\left(  x_{0}\right)  \ast\mathcal{A}%
+A^{T}\mathcal{H}X_{0},\\
\mathbf{\tilde{R}}  &  =\mathbf{r}^{T}\mathcal{H}X_{0}\mathcal{+}Df\left(
x_{0}\right)  \mathbf{R}+\left(  A\mathbf{b}+\mathbf{r}\right)  ^{T}%
\mathcal{H}\left(  \mathbf{b}^{T}\mathcal{A+}\mathbf{R}\right)  .
\end{align*}

\begin{remark}
\label{rem:Der-reminder}As in Remark \ref{rem:Num-reminder}, we see that
$\tilde{X}_{0}$ and $\widetilde{\mathcal{A}}$ can be computed accurately.
Moreover, if $\mathbf{r}$ and $\mathbf{R}$ are $O(\left\Vert \mathbf{b}%
\right\Vert ^{2})$ then $\mathbf{\tilde{R}}$ will also be $O(\left\Vert
\mathbf{b}\right\Vert ^{2})$.
\end{remark}

The representation (\ref{eq:Lohner-set}--\ref{eq:Lohner-der}) is intended to
control the wrapping effect for computing $f^{k}(x)$ and $Df^{k}(x)$ for
larger $\left\vert k\right\vert $. To compute a bound for $f^{k}(\mathbf{x})$
and $Df^{k}(\mathbf{x})$ in our representation, we start with a set enclosure
and with identity matrix:%
\begin{align}
\mathbf{x}  &  =x_{0}+\mathbf{b,}\label{eq:b-of-init-set}\\
X  &  =Id,\nonumber
\end{align}
meaning that we start with $A=Id,$ $\mathbf{r}=0,$ $X_{0}=Id,$ $\mathbf{R}=0$
and $\mathcal{A}=0$ in the representations of a set (\ref{eq:Lohner-set}) and
derivative (\ref{eq:Lohner-der}). We then iterate the map using
(\ref{eq:f-enclosure-rep}) and (\ref{eq:Df-enclosure-rep}). The resulting set
enclosure will be of the form%
\begin{equation}
f^{k}(\mathbf{x)}\subset x_{k}+A_{k}\mathbf{b}+\mathbf{r}_{k}\mathbf{,}
\label{eq:fk-set}%
\end{equation}
for some vector $x_{k}$, some matrix $A_{k}$, and some interval vector
$\mathbf{r}_{k}$, that follow from our iterative procedure.

\begin{remark}
The $\mathbf{b}$ in (\ref{eq:fk-set}) is the same as in
(\ref{eq:b-of-init-set}), which is the main objective of our representation.
By Remark \ref{rem:Num-reminder}, the $x_{k}$ and $A_{k}$ can be computed
accurately, and we can expect $\mathbf{r}_{k}$ to be small.
\end{remark}

The enclosure of the derivative at the end of our iterative procedure will be
\begin{equation}
Df^{k}(\mathbf{x)}\subset X_{k}+\mathbf{b}^{T}\mathcal{A}_{k}+\mathbf{R}_{k},
\label{eq:Dfk-set}%
\end{equation}
for some matrix $X_{k},$ some sequence of matrixes $\mathcal{A}_{k}$, and some
interval matrix $\mathbf{R}_{k}$, that follow from our iterative procedure.

\begin{remark}
The $\mathbf{b}$ in (\ref{eq:fk-set}) is the same as the one in
(\ref{eq:b-of-init-set}). By Remark \ref{rem:Der-reminder}, we can expect
$\mathbf{R}_{k}$ to be small, and the $X_{k}$, $\mathcal{A}_{k}$ can be
computed accurately.
\end{remark}

Above method of representing sets works quite nicely in our case, since the
Lomel\' i map is quadratic. The method though can also be applied in more
general setting:

\begin{remark}
To apply our method the map $f$ does not need to be quadratic. It is enough to
have interval enclosures $\mathbf{H}_{i}$ of the second derivatives on a set $U$, 
\[
\mathbf{H}_{i} = \left[ \left\{  A\in\mathbb{R}^{n}\times\mathbb{R}^{n}|A_{jk}%
\in\left[  \inf_{x\in U}\frac{\partial^{2}f_{i}}{\partial x_{j}\partial x_{k}%
}\left(  x\right)  ,\sup_{x\in U}\frac{\partial^{2}f_{i}}{\partial
x_{j}\partial x_{k}}\left(  x\right)  \right]  \right\}\right],
\]
for the set $U$ containing $x=x_{0}+Ab+r$. We can then use the $\mathbf{H}%
_{i}$ instead of $\nabla^{2}f_{i}$ in our computations. This gives a method
for the computation of bounds on iterates of maps and on their derivatives,
which reduces the wrapping effect.
\end{remark}

\section{Acknowledgements} We are grateful for useful discussions to J. L. Figueras, W. Tucker and D. Wilczak. We would like to thank the two anonymous reviewers for their suggestions and comments, which helped us to improve the manuscript.

\section{Appendix}

\begin{proof}
(of Lemma \ref{lem:Hess-manip}) We use the fact that%
\[
b^{T}A_{m}=\left(
\begin{array}
[c]{llll}%
\sum_{j}b_{j}A_{mj1} & \sum_{j}b_{j}A_{mj2} & \ldots & \sum_{j}b_{j}A_{mjn}%
\end{array}
\right)  ,
\]
hence the coefficient with index $ik$ in the matrix $b^{T}\mathcal{A}$ has the
form%
\begin{equation}
\left(  b^{T}\mathcal{A}\right)  _{mk}=\sum_{j}b_{j}A_{mjk}.
\label{eq:b-prod-hess-coeff}%
\end{equation}
This allows us to compute the left hand side of (\ref{eq:Hess-manip}) as (we
use (\ref{eq:b-prod-hess-coeff}) in the second equality, in the first we
simply multiply matrixes)
\[
\left(  B\left(  b^{T}\mathcal{A}\right)  \right)  _{ik}=\sum_{m}B_{im}\left(
b^{T}\mathcal{A}\right)  _{mk}=\sum_{m,j}B_{im}b_{j}A_{mjk}.
\]

To compute the right hand side of (\ref{eq:Hess-manip}), we first observe that
the $i$-th matrix in $B\ast\mathcal{A}$ is
\begin{equation}
\left(  B\ast\mathcal{A}\right)  _{i}=\sum_{m}B_{im}A_{m}.
\label{eq:tmp-hess-manip1}%
\end{equation}
We now compute the right hand side of (\ref{eq:Hess-manip}) as, (we use
(\ref{eq:b-prod-hess-coeff}) in the first equality, and
(\ref{eq:tmp-hess-manip1}) in the second equality)%
\begin{multline*}
\left(  b^{T}\left(  B\ast\mathcal{A}\right)  \right)  _{ik}=\sum_{j}%
b_{j}\left(  B\ast\mathcal{A}\right)  _{ijk}=\sum_{j}b_{j}\left(  \sum
_{m}B_{im}A_{m}\right)  _{jk}\\
=\sum_{j}b_{j}\sum_{m}\left(  B_{im}A_{m}\right)  _{jk}=\sum_{j,m}b_{j}%
B_{im}A_{mjk}.
\end{multline*}
Both sides of (\ref{eq:Hess-manip}) are the same, which finishes our proof.
\end{proof}

\bibliographystyle{unsrt}
\bibliography{papers}
\end{document}